\documentclass[11pt]{article}
\usepackage{amsmath, amssymb, array}
\usepackage{amsthm, fullpage}

\usepackage[boxed, linesnumbered, ruled, vlined]{algorithm2e}
\usepackage{graphicx}
\usepackage{tikz}
\usepackage{mathrsfs}
\usepackage[hang, flushmargin]{footmisc}
\usepackage{caption}

\usepackage{booktabs}
\usetikzlibrary{plotmarks}
\usepackage{array,multirow}

 \newtheorem{thm}{Theorem}[section]
 \newtheorem{cor}[thm]{Corollary}
 \newtheorem{lem}[thm]{Lemma}
 \newtheorem{claim}[thm]{Claim}
 \newtheorem{prop}[thm]{Proposition}

 \theoremstyle{definition}
 \newtheorem{defn}[thm]{Definition}
  \newtheorem{problem}[thm]{Problem}
 \newtheorem{exmp}[thm]{Example}

 \newtheorem{rem}[thm]{Remark}

\numberwithin{equation}{section}

 \newcommand{\s}{\mathcal{S}}
 \newcommand{\is}{\mathcal{I}}
 \newcommand{\vs}{\mathcal{V}}
 
 \newcommand{\set}[1]{\left\{#1\right\}}
 \newcommand{\Set}[2]{\set{#1\ \vert\ #2}}

\begin{document}

\baselineskip 10pt
\title{\bf Characterization and recognition of proper tagged probe interval graphs}
\author{
Sanchita Paul\thanks{Department of Mathematics, Jadavpur University, Kolkata - 700 032, India. sanchitajumath@gmail.com},\ \ 
Shamik Ghosh\thanks{Department of Mathematics, Jadavpur University, Kolkata - 700 032, India. sghosh@math.jdvu.ac.in}\ \thanks{(Corresponding author)},\ \ 
Sourav Chakraborty\thanks{Chennai Mathematical Institute, Chennai, India. sourav@cmi.ac.in} \ \ 
 and \ 
Malay Sen\thanks{Department of Mathematics, North Bengal University, District - Darjeeling, West Bengal, Pin - 734 430, India.\ \ \  senmalay10@gmail.com}
}
\date{}


\maketitle

\begin{abstract}
\noindent
\footnotesize {Interval graphs were used in the study of the human genome project by the famous molecular biologist Benzer. Later on probe interval graphs were introduced by Zhang as a generalization of interval graphs for the study of cosmid contig mapping of DNA. Further research in this area required more useful and cost effective tools. The concept of {\em tagged probe interval graphs} (briefly, TPIG) is motivated from this point of view,  where the set of vertices is partitioned into two sets, namely, probes and nonprobes and there is an interval on the real line corresponding to each vertex. The graph has an edge between two probe vertices if their corresponding intervals intersect, has an edge between a probe vertex and a nonprobe vertex if the interval corresponding to a nonprobe vertex contains at least one endpoint of the interval corresponding to a probe vertex and the set of nonprobe vertices is an independent set. 

\vspace{1em}\noindent
In this paper, we consider a natural subclass of TPIG, namely, the class of {\em proper tagged probe interval graphs} (in short PTPIG). We present characterization theorem and a linear time recognition algorithm for PTPIG. To obtain this characterization theorem we introduce {\em canonical sequences} for proper interval graphs. In the recognition algorithm for PTPIG, we introduce and solve a variation of consecutive $1$'s problem, namely, {\em oriented consecutive $1$'s problem} and some variations of PQ-tree algorithm. We also discuss the interrelations between the classes of PTPIG and TPIG with probe interval graphs and probe proper interval graphs.}
\smallskip

\end{abstract}


\baselineskip 18pt

\vspace{1em}

\newpage

\section{Introduction}

One of the most intriguing problem in molecular biology, especially in the human genome project, is the physical mapping of DNA that aims to reconstruct the relative position of fragments of DNA along the genome. In 1959, the famous molecular biologist Benzer \cite{B} nicely used the model of interval graphs to obtain such a physical map from information on pairwise overlaps of the fragments. Fragments are replicated as clones, their fingerprints are obtained by certain technique and finally overlap information is determined when their fingerprints are sufficiently similar. The interval graph is formed with clones as vertices and the adjacency represents overlap. Since then the study of interval graphs was spearheaded. Over the years the graph class established itself as a very rich and most popular subclass of the class of chordal graphs and in general, perfect graphs irrespective of its application part. In 1964, Gilmore and Hoffman \cite{GH} characterized interval graphs as a chordal graph whose complement is transitively oriented. In 1975, Booth and Leuker \cite{BL75} gave the first linear time algorithm for recognition of interval graphs by introducing \textrm{PQ}-tree algorithm.

\vspace{1em}\noindent
Various other characterizations, linear time recognition algorithms of interval graphs and its most widely studied subclass, proper interval graphs during the last fifty years and the research is still continuing. In 1994, Zhang introduced the concept of probe interval graphs which worked successfully as a model for a new concept, known as cosmid contig mapping. This generates overlap information by hybridization. A set of clones is placed on a filter for colony hybridization and the filter is probed with clones which are labeled. Thus the set of vertices (clones) is partitioned into probes and and nonprobes and the adjacency requires the overlap information between a pair of clones only when one of them is a probe. A polynomial time recognition algorithm for probe interval graphs was obtained by Johnson and Spinrad \cite{JS} in 2001 and the first linear time one was presented by McConnell and Nussbaum \cite{MN} in 2009. In the year 2010, Ghosh, Podder and Sen \cite{SMM} obtained a characterization of probe interval graphs in terms of adjacency matrix. Recently, Nussbaum \cite{Nu} gave a linear time algorithm for probe proper interval graphs in 2014.

\vspace{1em}\noindent
In a short time after the introduction of probe interval graphs, development of research in molecular biology requires further refinements and the concept of tagged probe interval graphs are defined in \cite{LP,SWZ,LP2}. In this model, a set of clones (probes) are radioactively labeled at their ends and one detects the overlapping of a pair of clones when one of them is labeled and the other contains at least one end of the labeled one. Thus here also vertices are partitioned into probes and nonprobes but the adjacency is determined by the above mentioned overlapping information.

\vspace{1em}\noindent
Interestingly, since its inception, there is still no characterization theorem or recognition algorithm for tagged probe interval graphs or any natural subclass of the class of tagged probe interval graphs, excepting probe proper interval graphs (cf. Subsection~\ref{secex}). The problem seems to be hard and a challenging one. However, it is proved that, like probe interval graphs, tagged probe interval graphs are also weakly chordal and hence perfect.

\vspace{1em}\noindent
In this paper, we obtain a characterization theorem and a linear time recognition algorithm for a natural subclass of the class of tagged probe interval graphs, namely, proper tagged probe interval graphs. To obtain this characterization theorem we introduce a concept called  canonical sequence for proper interval graphs, which we believe would be of independent interest in the study of proper interval graphs. Also to obtain the recognition algorithm for proper tagged probe interval graphs, we introduce and solve a variation of consecutive $1$'s problem, namely, oriented consecutive $1$'s problem and some variations of PQ-tree algorithm. We hope these results will be beneficiary to molecular biologists involved in the human genome project. For more detailed information on applications of interval graphs, probe interval graphs and tagged probe interval graphs in molecular biology and in other areas one may consult \cite{Br,BLS,G,GT,LP,Z}.

\section{Preliminaries}

A graph $G = (V, E)$ is an \textit{interval graph} if one can map each vertex into an interval on the real line so that any two vertices are adjacent if and only if their corresponding intervals intersect. Such a mapping of vertices into an interval on the real line is called an \textit{interval representation of $G$}. 

\vspace{1em}\noindent
A graph $G = (V,E)$ is a \textit{probe interval graph} (in short, PIG) if the vertex set $V$ can be partitioned into two disjoint sets probe vertices $P$ and nonprobe vertices $N$ and one can map each vertex into an interval on the real line (vertex $x\in V$ mapped to $I_x$) such that there is an edge between two vertices $x$ and $y$ if and only if at least one of them is in $P$ and their corresponding intervals $I_x$ and $I_y$ intersect. When the interval representation is proper (i.e., no interval is properly contained in another interval) the graph is called a {\em probe proper interval graph} (briefly, PPIG). Clearly PPIG is a subclass of PIG.

\begin{defn}\label{defn1} \cite{LP}
A graph $G = (V,E)$ is a \textit{tagged probe interval graph} if the vertex set $V$ can be partitioned into two disjoint sets $P$ (called ``probe vertices") and $N$ (called ``nonprobe vertices") and one can map each vertex into an interval on the real line (vertex $x\in V$ mapped to $I_x = [\ell_x, r_x]$) such that all the following conditions hold:
\begin{enumerate}
\item $N$ is an independent set in $G$, i.e., there is no edge between nonprobe vertices. 
\item If $x, y \in P$, then there is an edge between $x$ and $y$ if and only if $I_x \cap I_y \neq \emptyset$, or in other words the mapping is an interval representation of the subgraph of $G$ induced by $P$. 
\item If $x\in P$ and $y\in N$, then there is an edge between $x$ and $y$ if and only if the interval corresponding to the nonprobe vertex contains at least one endpoint of the interval corresponding to the probe vertex, i.e., either $\ell_x \in I_y$ or $r_x \in I_y$. 
\end{enumerate}

\noindent
We call the collection $\Set{I_x}{x\in V}$ a {\em TPIG representation} of $G$. If the partition of the vertex set $V$ into probe and nonprobe vertices is given, then we denote the graph as $G = (P,N,E)$. 
\end{defn}

\begin{problem}
Given a graph $G = (P, N, E)$, give a linear time algorithm for checking if $G$ is a  \textit{tagged probe interval graph}.
\end{problem}

\noindent
A natural and well studied subclass of interval graphs is the class of proper interval graphs. A {\em proper interval graph} $G$ is an interval graph in which there is an interval representation of $G$ such that no interval contains another properly. Such an interval representation is called a \textit{proper interval representation of $G$}.

\vspace{1em}\noindent
In this paper, we study a natural special case of tagged probe interval graphs which we call \textit{proper tagged probe interval graph} (in short, \textit{PTPIG}). The only extra condition that a PTPIG should satisfy over TPIG is that the mapping of the vertices into intervals on the real line that gives a TPIG representation of $G$ should be a proper interval representation of the subgraph of $G$ induced by $P$. 

\subsection{Notations}\label{subnote}

Suppose a graph $G$ is a PTPIG (or TPIG), then we will assume that the vertex set is partitioned into two sets $P$ (for probe vertices) and $N$ (for nonprobe vertices). To indicate that the partition is known to us, we will sometimes denote $G$ by $G=(P,N,E)$, where $E$ is the edge set. We will denote by $G_P$ the subgraph of $G$ that is induced by the vertex set $P$.  We will assume that there are $p$  probe vertices $\{u_1, \dots, u_p\}$ and $q$ nonprobe vertices $\{w_1, \dots, w_q\}$. To be consistent in our notation we will use $i$ or $i'$ or $i_1, i_2, \dots$ as indices for probe vertices and use $j$ or $j'$ or $j_1, j_2, \dots$ as indices for nonprobe vertices. 

\vspace{1em}\noindent
Let $G=(V,E)$ be a graph and $v\in V$. Then the set $N[v]=\Set{u\in V}{u\text{ is adjacent to } v}\cup\set{v}$ is the {\em closed neighborhood} of $v$ in $G$. A graph $G$ is called {\em reduced} if no two vertices have the same closed neighbourhood. If the graph is not reduced then we define an equivalence relation on the vertex set $V$ such that $v_i$ and $v_j$ are equivalent if and only if $v_i$ and $v_j$ have the same (closed) neighbors in $V$.  Each equivalence class under this relation is called a {\em block} of $G$. For any vertex $v\in V$ we denote the equivalence class containing $v$ by $B(v)$. Let the blocks of $G$ be $B^G_1, B^G_2,\ldots, B^G_t$ (or sometimes just $B_1, B_2,\ldots, B_t$). So, the collection of blocks is a partition of $V$.  The reduced graph of $G$ (denoted by $\widetilde{G}=(\widetilde{V},\widetilde{E})$) is the graph obtained by merging all the vertices that are in the same equivalence class. 

\vspace{1em}\noindent
If $M$ is a $(0,1)$-matrix, then we say $M$ satisfies the {\em consecutive $1$'s property} if in each row and column, $1$'s appear consecutively \cite{H,MPT}. We will denote by $A(G)$ the {\em augmented adjacency matrix} of the graph $G$, in which all the diagonal entries are $1$ and non-diagonal elements are same as the adjacency matrix of $G$.  

\subsection{Background Materials}

\subsubsection{PQ-trees}

In the past few decades many variations of interval graphs have been studied mostly in context of modeling different scenarios arising from molecular biology and DNA sequencing. Understanding the structure and properties of these classes of graphs and designing efficient recognition algorithms are the central problems in this field. Many times these studies have led to nice combinatorial problems and development of important data structures. 

\vspace{1em}\noindent
For example, the original linear time recognition algorithm for interval graphs by Booth and Lueker~\cite{BL76} in 1976 is based on their complex PQ-tree data structure (also see \cite{BL75}). Habib et al.~\cite{HMPV} in 2000 showed how to solve the problem more simply using lexicographic breadth-first search, based on the fact that a graph is an interval graph if and only if it is chordal and its complement is a comparability graph. A similar approach using a 6-sweep LexBFS algorithm is described in Corneil, Olariu and Stewart~\cite{cor3} in 2010. 

\vspace{1em}\noindent
In this paper, we will be using the data structure of PQ-trees quite extensively. PQ-trees are not only used to check whether a given matrix satisfies the consecutive $1$'s property; they also store all the possible permutations such that if one permutes the rows (or column) using the permutations, the matrix would satisfy the consecutive $1$'s property. We define a generalization of the problem of checking consecutive $1$'s property to a problem called oriented-consecutive $1$'s problem and show how the PQ-tree representation can be used to solve this problem also. Details are available in Section~\ref{reco}. 

\begin{table}
\centering                           
\begin{tabular}{|c|l|l|}
\hline
Index & Properties & References\\
\hline
1 & $G$ is a Proper Interval Graph.  & \multirow{3}{*}{\cite{G,rob1,rob2,rob3,rob4}} \\ \cline{1-2}
2 & $G$ is a Unit Interval Graph &\\ \cline{1-2}
3 & $G$ is {\em claw-free}, i.e., $G$ does not contain $K_{1,3}$ as an induced subgraph. &\\ \hline
4 & For all $v\in V$, elements of $N[v]=\Set{u\in V}{uv\in E}\cup\set{v}$ are & \multirow{8}{*}{\cite{cor1,cor2,cor3,G}}  \\
  & consecutive for some ordering of $V$ (closed neighborhood condition). & \\ \cline{1-2}
5 & There is an ordering ${v_{1},v_{2},\cdots,v_{n}}$ of V such that &\\
  & $G$ has a proper interval graph representation &\\
  & $\{I_{v_{i}}=[a_{i},b_{i}]|i=1,2,\cdots,n\}$ where $a_{1}<a_{2}<\cdots<a_{n}$ &\\
  &  and $b_{1}<b_{2}<\cdots<b_{n}$. &\\ \cline{1-2}
6 & There is an ordering of V such that the augmented &\\
  & adjacency matrix $A(G)$ of $G$ satisfies the consecutive $1$'s property. &\\ \hline
7 & \footnotesize{\textit{A {\em straight enumeration} of $G$ is a linear ordering of blocks}} & \multirow{5}{*}{\cite{DHH,HH,HSS,Nu}}\\
  & \footnotesize{\textit{(vertices having same closed neighborhood) in $G$, such that}} & \\
	& \footnotesize{\textit{for every block, the block and its neighboring blocks are}} & \\
	& \footnotesize{\textit{consecutive in the ordering.}} & \\
	& $G$ has a straight enumeration & \\
	& which is unique up to reversal, if $G$ is connected. & \\ \hline   
8 & \footnotesize{\textit{The {\em reduced graph} $\widetilde{G}$ is obtained from $G$}} & \multirow{6}{*}{\cite{MSG}}\\
  & \footnotesize{\textit{by merging vertices having same closed neighborhood.}} & \\
	& \footnotesize{\textit{$G(n,r)$ is a graph with $n$ vertices $x_1,x_2,\ldots,x_n$ such that}} & \\
	& \footnotesize{\textit{$x_i$ is adjacent to $x_j$ if and only if $0<|i-j|\leqslant r$, where $r$ is a positive integer.}} \\
  & $\widetilde{G}$ is an induced subgraph of $G(n,r)$ & \\
  & for some positive integers $n,r$ with $n>r$.& \\ 
\hline
\end{tabular}
\caption{\footnotesize{Characterizations of proper interval graphs: equivalent conditions on an interval graph $G=(V,E)$.}}\label{t:proper}
\end{table}

\subsubsection{Proper Interval Graphs}\label{sppg}

Recall that a proper interval graph $G$ is an interval graph in which there is an interval representation of $G$ such that no interval contains another properly. It is important to note that a proper interval graph $G$ may have an interval representation which is not proper. Linear-time recognition algorithms for proper interval graphs are obtained in \cite{DHH,FMM,HH,HSS}. A {\em unit interval graph} is an interval graph in which there is an interval representation of $G$ such that all intervals have the same length. Interestingly, these two concepts are equivalent. Another equivalence is that an interval graph is a proper interval graph if and only if it does not conatin $K_{1,3}$ as an induced subgraph. The class of proper interval graphs is an extremely rich class of graphs and there are several other characterizations of it (see Table~\ref{t:proper}). Among them we repeatedly use the following equivalent conditions in the rest of the paper: 

\begin{thm}\label{troberts}
Let $G=(V,E)$ be an interval graph. then the following are equivalent:
\begin{enumerate}
\item $G$ is a proper interval graph.
\item There is an ordering of $V$ such that for all $v\in V$, elements of $N[v]$ are consecutive (the closed neighborhood condition).\label{cnc}
\item There is an ordering of $V$ such that the augmented adjacency matrix $A(G)$ of $G$ satisfies the consecutive $1$'s property.\label{con}
\item There is an ordering $\set{v_1,v_2,\ldots,v_n}$ of $V$ such that $G$ has a proper interval representation $\Set{I_{v_i}=[a_i,b_i]}{i=1,2,\ldots,n}$ where $a_i\neq b_j$ for all $i,j\in\set{1,2,\ldots,n}$, $a_1<a_2<\cdots <a_n$ and $b_1< b_2<\cdots <b_n$.\label{order}
\end{enumerate}
\end{thm} 

\begin{rem}\label{remcan}
We note that in a proper interval graph $G=(V,E)$, the ordering of $V$ that satisfies any one of the conditions (\ref{cnc}), (\ref{con}) and (\ref{order}) in the above theorem also satisfies the other conditions among them.
\end{rem}

\subsubsection{Relation between PTPIG and other variants}\label{secex}

The definition of PIG is very similar to that of TPIG, but the two classes of graphs are not comparable~\cite{LP}. But PPIG is a proper subclass of PTPIG. In fact, since in an interval representation of PPIG, no interval contains other properly, there is an edge between a probe and nonprobe vertices if and only if an endpoint of the interval corresponding to the probe vertex belongs to the interval corresponding the nonprobe vertex. 

\vspace{0.5em}\noindent
On the other hand, $K_{1,3}$ with a single nonprobe at the center cannot be a PPIG for otherwise it would be a proper interval graph (as any probe interval graph with a single nonprobe vertex is an interval graph). But it is a PTPIG by choosing three disjoint intervals for probe vertices and an interval containing all of them corresponding to the nonprobe vertex. As $K_{1,3}$ is an interval graph, it is an example of PIG and PTPIG, but not a PPIG. 

\vspace{0.5em}\noindent
Similarly, $C_4$ with a single nonprobe vertex is a PTPIG with an interval representation $[3,4]$ for the nonprobe and $\set{[1,3],[2,5],[4,6]}$ for probes, but this is not a PIG (for otherwise it would be an interval graph). Now consider the graph $G_1$ in Figure~\ref{fig:ex5}. That $G_1$ is a PIG and TPIG follows from the interval representation described in Figure~\ref{ex5}. But $G_1$ is not a PTPIG as the subgraph induced by probe vertices is $K_{1,3}$ which is not a proper interval graph.

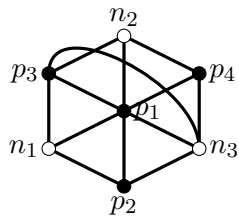
\begin{figure}[b]
\centering
\begin{tikzpicture}
\draw[-][draw=black,very thick] (0,0) -- (0,1);
\draw[-][draw=black,very thick] (0,0) -- (1,.5);
\draw[-][draw=black,very thick] (0,0) -- (-1,.5);
\draw[-][draw=black,very thick] (0,0) -- (1,-.5);
\draw[-][draw=black,very thick] (0,0) -- (-1,-.5);
\draw[-][draw=black,very thick] (0,0) -- (0,-1);
\draw[-][draw=black,very thick] (0,1) -- (-1,.5);
\draw[-][draw=black,very thick] (0,1) -- (1,.5);
\draw[-][draw=black,very thick] (-1,.5) -- (-1,-.5);
\draw[-][draw=black,very thick] (-1,-.5) -- (0,-1);
\draw[-][draw=black,very thick] (0,-1) -- (1,-.5);
\draw[-][draw=black,very thick] (1,-.5) -- (1,.5);
\draw[-][draw=black,very thick] (-1,0.5) to [out=90,in=90] (1,-.5);
\draw [fill=black] (0,0) circle [radius=0.09];
\draw [fill=black] (1,0.5) circle [radius=0.09];
\draw [fill=black] (-1,0.5) circle [radius=0.09];
\draw [fill=black] (0,-1) circle [radius=0.09];
\draw [fill=white] (0,1) circle [radius=0.09];
\draw [fill=white] (-1,-.5) circle [radius=0.09];
\draw [fill=white] (1,-.5) circle [radius=0.09];
\node [above] at (0,1) {{$n_{2}$}};
\node [right] at (1,.5) {{$p_{4}$}};
\node [right] at (0,0) {{$p_{1}$}};
\node [right] at (1,-.5) {{$n_{3}$}};
\node [below] at (0,-1) {{$p_{2}$}};
\node [left] at (-1,-.5) {{$n_{1}$}};
\node [left] at (-1,.5) {{$p_{3}$}};
\end{tikzpicture}
\caption{The graph $G_1$.} \label{fig:ex5}
\end{figure}

\begin{figure}[t]
\centering
\begin{tikzpicture}
\draw[-][draw=black,very thick] (1,0) -- (4,0);
\draw[-][draw=black,thick] (2,-0.3) -- (5,-0.3);
\draw[-][draw=black,very thick] (3,-.6) -- (7,-.6);
\draw[-][draw=black,thick] (3.5,-.9) -- (8,-.9);
\draw[-][draw=black,very thick] (4.5,-1.2) -- (6,-1.2);
\draw[-][draw=black,thick] (5.8,-1.5) -- (10,-1.5);
\draw[-][draw=black,very thick] (6.5,-1.8) -- (11,-1.8);
\node [left] at (.8,0) {\tiny{$p_{2}$}};
\node [left] at (.8,-.3) {\tiny{$n_{1}$}};
\node [left] at (.8,-.6) {\tiny{$p_{1}$}};
\node [left] at (.8,-.9) {\tiny{$n_{3}$}};
\node [left] at (.8,-1.2) {\tiny{$p_{3}$}};
\node [left] at (.8,-1.5) {\tiny{$n_{2}$}};
\node [left] at (.8,-1.8) {\tiny{$p_{4}$}};
\node at (1,0) {\pgfuseplotmark{square*}};
\node at (4,0) {\pgfuseplotmark{square*}};
\node at (3,-.6) {\pgfuseplotmark{square*}};
\node at (7,-.6) {\pgfuseplotmark{square*}};
\node at (4.5,-1.2) {\pgfuseplotmark{square*}};
\node at (6,-1.2) {\pgfuseplotmark{square*}};
\node at (11,-1.8) {\pgfuseplotmark{square*}};
\node at (6.5,-1.8) {\pgfuseplotmark{square*}};
\node at (1,.3) {\tiny{$1$}};
\node at (1.3,.3) {\tiny{$2$}};
\node at (1.6,.3) {\tiny{$3$}};
\node at (1.9,.3) {\tiny{$4$}};
\node at (2.1,.3) {\tiny{$5$}};
\node at (2.4,.3) {\tiny{$6$}};
\node at (2.7,.3) {\tiny{$7$}};
\node at (3,.3) {\tiny{$8$}};
\node at (3.3,.3) {\tiny{$9$}};
\draw[densely dotted][draw=black,thick] (3.6,0.3) -- (10.8,.3);
\node at (11.1,.3) {\tiny{$32$}};
\end{tikzpicture}
\caption{Interval representation of the graph $G_1$ in Figure ~\ref{fig:ex5}.}\label{ex5}
\end{figure}
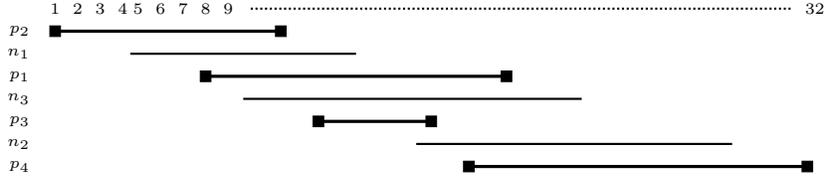

\section{Canonical Sequence of Proper Interval Graphs}\label{sproper}

Let $G$ be a proper interval graph. Then there is an ordering of $V$ that satisfies conditions (\ref{cnc}), (\ref{con}) and (\ref{order}) of Theorem~\ref{troberts}. Henceforth we call such an ordering, a {\em natural} or {\em canonical} ordering of $V$. But this canonical ordering is not unique. A proper interval graph may have more than one canonical ordering of its vertices. Interestingly, it follows from Corollary 2.5 of \cite{DHH} (also see \cite{Nu}) that the canonical ordering is unique up to reversal for a connected reduced proper interval graph.

\begin{defn}\label{dseq}
Let $G=(V,E)$ be a proper interval graph. Let $\set{v_1,v_2,\ldots,v_n}$ be a canonical ordering of the set $V$ with the interval representation be $\Set{I_{v_i}=[a_i,b_i]}{i=1,2,\ldots,n}$ where $a_i\neq b_j$ for all $i,j\in\set{1,2,\ldots,n}$, $a_1<a_2<\cdots <a_n$ and $b_1< b_2<\cdots <b_n$. Now we combine all $a_i$ and $b_i$ ($i=1,2,\ldots,n$) in a single increasing sequence which we call the {\em interval canonical sequence} with respect to the canonical ordering of vertices of $G$ and is denoted by $\is_G$. 

\vspace{1em}\noindent
Now if we replace $a_i$ or $b_i$ by $i$ for all $i=1,2,\ldots,n$ in $\is_G$, then we obtain a sequence of integers belonging to $\set{1,2,\ldots,n}$ each occurring twice. We call such a sequence a {\em canonical sequence} of $G$ with respect to the canonical ordering of vertices of $G$ and is denoted by $\s_G$. Moreover, if we replace $i$ by $v_i$ for all $i=1,2,\ldots,n$ in $\s_G$, then the resulting sequence is called a {\em vertex canonical sequence} of $G$ (corresponding to the canonical sequence $\s_G$) and is denoted by $\vs_G$. 

\vspace{1em}\noindent
Note that $\s_G$ and its corresponding $\vs_G$ and $\is_G$ can all be obtained uniquely from each other. Thus by abuse of notations, sometimes we will use the term canonical sequence to mean any of these. 
\end{defn}

\subsection{Example of Canonical Sequence}

\begin{exmp}\label{ex22}
Consider the proper interval graph $G=(V,E)$ whose augmented adjacency matrix $A(G)$ along with a proper interval representation is given in Table \ref{tap2}. Note that vertices of $G$ are arranged in a canonical ordering as $A(G)$ that satisfies the consecutive $1$'s property.
Let $[a_i,b_i]$ be the interval corresponding to the vertex $v_i$ for $i=1,2,\ldots,8$. Then 
$$\begin{array}{|c|c|c|c|c|c|c|c|}
\hline
a_1 & a_2 & a_3 & a_4 & a_5 & a_6 & a_7 & a_8 \\
\hline
2 & 3 & 6 & 7 & 8 & 10 & 16 & 22 \\
\hline
b_1 & b_2 & b_3 & b_4 & b_5 & b_6 & b_7 & b_8 \\
\hline
4 & 9 & 12 & 14 & 17 & 20 & 24 & 28 \\
\hline
\end{array}$$
Then the interval canonical sequence combining $a_i$ and $b_i$ is given by 
$$\is_G=(a_1, a_2, b_1, a_3, a_4, a_5, b_2, a_6, b_3, b_4, a_7, b_5, b_6, a_8, b_7, b_8).$$
Therefore the canonical sequence and the vertex canonical sequence with respect to the given canonical vertex ordering is 
$$\s_G=(1\ 2\ 1\ 3\ 4\ 5\ 2\ 6\ 3\ 4\ 7\ 5\ 6\ 8\ 7\ 8)\text{ and } \vs_G=(v_1\ v_2\ v_1\ v_3\ v_4\ v_5\ v_2\ v_6\ v_3\ v_4\ v_7\ v_5\ v_6\ v_8\ v_7\ v_8).$$ 
\end{exmp}

\begin{table}[t]
{\footnotesize
$$\begin{array}{cc|cccccccc|}
& \multicolumn{1}{c}{} & [2,4] & [3,9] & [6,12] & [7,14] & [8,17] & [10,20] & [16,24] & \multicolumn{1}{c}{[22,28]}\\
& \multicolumn{1}{c}{} & v_1 & v_2 & v_3 & v_4 & v_5 & v_6 & v_7 & \multicolumn{1}{c}{v_8} \\
\cline{3-10}
\text{$[2,4]$} & v_1 & 1 & 1 & 0 & 0 & 0 & 0 & 0 & 0 \\
\text{$[3,9]$} & v_2 & 1 & 1 & 1 & 1 & 1 & 0 & 0 & 0 \\
\text{$[6,12]$} & v_3 & 0 & 1 & 1 & 1 & 1 & 1 & 0 & 0 \\
\text{$[7,14]$} & v_4 & 0 & 1 & 1 & 1 & 1 & 1 & 0 & 0 \\
\text{$[8,17]$} & v_5 & 0 & 1 & 1 & 1 & 1 & 1 & 1 & 0 \\
\text{$[10,20]$} & v_6 & 0 & 0 & 1 & 1 & 1 & 1 & 1 & 0 \\
\text{$[16,24]$} & v_7 & 0 & 0 & 0 & 0 & 1 & 1 & 1 & 1 \\
\text{$[22,28]$} & v_8 & 0 & 0 & 0 & 0 & 0 & 0 & 1 & 1 \\
\cline{3-10}
\end{array}$$}
\caption{The augmented adjacency matrix $A(G)$ of the graph $G$ in Example \ref{ex22}}\label{tap2}
\end{table}

\subsection{Structure of the Canonical Sequence for Proper Interval Graphs}

If a graph $G$ is a reduced proper interval graph then the following lemma states that the canonical sequence for $G$ is unique up to reversal. 

\begin{lem}\label{obsseq}
Let $G=(V,E)$ be a proper interval graph and $V=\set{v_1,v_2,\ldots,v_n}$ be a canonical ordering of vertices of $G$. Then the canonical sequence $\s_G$ is independent of proper interval representations that satisfies the given canonical ordering. Moreover $\s_G$ is unique up to reversal for connected reduced proper interval graphs.
\end{lem}

\begin{proof}
Let $\Set{I_{v_i}=[a_i,b_i]}{i=1,2,\ldots,n}$ and $\Set{J_{v_i}=[c_i,d_i]}{i=1,2,\ldots,n}$ be two proper interval representations of $G$ that satisfy the given canonical ordering. Now for any $i<j$, $a_j< b_i$ if and only if $v_i$ is adjacent to $v_j$ if and only if $c_j< d_i$. Thus the canonical sequence $\s_G$ is independent of proper interval graph representations. Now since the canonical ordering is unique up to reversal for a connected reduced proper interval graph, the canonical sequence $\s_G$ is unique up to reversal for connected reduced proper interval graphs.
\end{proof}

\noindent
Now there is an alternative way to get the canonical sequence directly from the augmented adjacency matrix. Let $G=(V,E)$ be a proper interval graph with $V=\Set{v_i}{i=1,2,\ldots,n}$ and $A(G)$ be the augmented adjacency matrix of $G$ with consecutive $1$'s property. We partition positions of $A(G)$ into two sets $(L,U)$ by drawing a polygonal path from the upper left corner to the lower right corner such that the set $L$ [resp. $U$] is closed under leftward or downward [respectively, rightward or upward] movement (called a {\em stair partition} \cite{BDGS,SDRW,SDW}) and $U$ contains precisely all the zeros right to the principal diagonal of $A(G)$ (see Table \ref{stseq}(left)). This is possible due to the consecutive $1$'s property of $A(G)$. Now we obtain a sequence of positive integers belonging to $\set{1,2,\ldots,n}$, each occurs exactly twice, by writing the row or column numbers as they appear along the stair. We call this sequence, the {\em stair sequence} of $A(G)$ (see Table \ref{stseq}(right)) and note that it is same as the canonical sequence of $G$ with respect to the given canonical ordering of vertices of $G$. 

\begin{table}[t]
{\footnotesize
$$\begin{array}{c|cc|ccc|c|c|c|}
 \multicolumn{1}{c}{} & v_1 & \multicolumn{1}{c}{v_2} & v_3 & v_4 & \multicolumn{1}{c}{v_5} & \multicolumn{1}{c}{v_6} & \multicolumn{1}{c}{v_7} & \multicolumn{1}{c}{v_8} \\
\cline{2-9}
 v_1 & 1 & 1 & 0 & 0 & \multicolumn{1}{c}{0} & \multicolumn{1}{c}{0} & \multicolumn{1}{c}{0} & 0 \\
\cline{4-6}
 v_2 & 1 & \multicolumn{1}{c}{1} & 1 & 1 & 1 & \multicolumn{1}{c}{0} & \multicolumn{1}{c}{0} & 0 \\
\cline{7-7}
 v_3 & 0 & \multicolumn{1}{c}{1} & 1 & 1 & \multicolumn{1}{c}{1} & 1 & \multicolumn{1}{c}{0} & 0 \\
 v_4 & 0 & \multicolumn{1}{c}{1} & 1 & 1 & \multicolumn{1}{c}{1} & 1 & \multicolumn{1}{c}{0} & 0 \\
\cline{8-8}
 v_5 & 0 & \multicolumn{1}{c}{1} & 1 & 1 & \multicolumn{1}{c}{1} & \multicolumn{1}{c}{1} & 1 & 0 \\
 v_6 & 0 & \multicolumn{1}{c}{0} & 1 & 1 & \multicolumn{1}{c}{1} & \multicolumn{1}{c}{1} & 1 & 0 \\
\cline{9-9}
 v_7 & 0 & \multicolumn{1}{c}{0} & 0 & 0 & \multicolumn{1}{c}{1} & \multicolumn{1}{c}{1} & \multicolumn{1}{c}{1} & 1 \\
 v_8 & 0 & \multicolumn{1}{c}{0} & 0 & 0 & \multicolumn{1}{c}{0} & \multicolumn{1}{c}{0} & \multicolumn{1}{c}{1} & 1 \\
\cline{2-9}
\end{array}\hspace{0.5in} \begin{array}{c|cc|ccc|c|c|c|c}
\multicolumn{1}{c}{} & v_1 & \multicolumn{1}{c}{v_2} & v_3 & v_4 & \multicolumn{1}{c}{v_5} & \multicolumn{1}{c}{v_6} & \multicolumn{1}{c}{v_7} & \multicolumn{1}{c}{v_8} & \\
\cline{2-9}
 v_1 & 1 & 2 & 1 &  & \multicolumn{1}{c}{} & \multicolumn{1}{c}{} & \multicolumn{1}{c}{} & & \\
\cline{4-6}
 v_2 &  & \multicolumn{1}{c}{} & 3 & 4 & 5 & \multicolumn{1}{c}{2} & \multicolumn{1}{c}{} &  & \\
\cline{7-7}
v_3 &  & \multicolumn{1}{c}{} &  &  & \multicolumn{1}{c}{} & 6 & \multicolumn{1}{c}{3} &  & \\
v_4 &  & \multicolumn{1}{c}{} &  &  & \multicolumn{1}{c}{} &  & \multicolumn{1}{c}{4} &  &\\
\cline{8-8}
v_5 &  & \multicolumn{1}{c}{} &  &  & \multicolumn{1}{c}{} & \multicolumn{1}{c}{} & 7 & 5 &\\
v_6 &  & \multicolumn{1}{c}{} &  &  & \multicolumn{1}{c}{} & \multicolumn{1}{c}{} &  & 6 &\\
\cline{9-9}
v_7 &  & \multicolumn{1}{c}{} &  &  & \multicolumn{1}{c}{} & \multicolumn{1}{c}{} & \multicolumn{1}{c}{} & 8 & 7 \\
v_8 &  & \multicolumn{1}{c}{} &  & & \multicolumn{1}{c}{} & \multicolumn{1}{c}{} & \multicolumn{1}{c}{} &  & 8\\
\cline{2-9}
\end{array}$$}
\caption{The matrix $A(G)$ with its stair partition and the stair sequence $(1\ 2\ 1\ 3\ 4\ 5\ 2\ 6\ 3\ 4\ 7\ 5\ 6\ 8\ 7\ 8)$ of $A(G)$ of the graph $G$ in Example \ref{ex22}.}\label{stseq}
\end{table}

\begin{prop}\label{staircanseq}
Let $G=(V,E)$ be a proper interval graph with a canonical ordering $V=\set{v_1,v_2,\ldots,v_n}$ of vertices of $G$. Let $A(G)$ be the augmented adjacency matrix of $G$ arranging vertices in the same order as in the canonical ordering. Then the canonical sequence $\s_G$ of $G$ is the same as the stair sequence of $A(G)$.
\end{prop}

\begin{proof}
We show that the stair sequence is a canonical sequence for some proper interval representation and so the proof follows by the uniqueness mentioned in Lemma \ref{obsseq}. Given the matrix $A(G)$, a proper interval representation of $G$ is obtained as follows. Let $a_i=i$ and $b_i=U(i)+1-\frac{1}{i}$ where $U(i)=\max\Set{j}{j\geq i\text{ and }v_iv_j=1\text{ in }A(G)}$ for each $i=1,2,\ldots,n$. Then $\Set{I_{v_i}=[a_i,b_i]}{i=1,2,\ldots,n}$ is a proper interval representation of $G$.~\cite{DHH,HH} If $U(1)>1$, to make all the endpoints distinct, we slightly increase the value of $b_1$ (which is the only integer valued right endpoint and is equal to $a_{U(1)}$) so that it is still less than its nearest endpoint which is greater than it. Now we get a proper interval representation of $G$ that satisfies the condition \ref{order} of Theorem \ref{troberts}. Then the canonical sequence merges with the stair sequence for this proper interval representation of $G$ as for $i<j$, $a_j=j<b_i$ if and only if $v_iv_j=1$ if and only if the column number $j$ appears before the row number $i$ in the stair sequence.
\end{proof}

\begin{cor}\label{corunique}
Let $G=(V,E)$ be a connected proper interval graph. Then $\s_G$ is unique up to reversal.
\end{cor}

\begin{proof}
By Lemma \ref{obsseq}, the result is true if $G$ is reduced. Suppose $G$ is not reduced and $\widetilde{G}=(\widetilde{V},\widetilde{E})$ be the reduced graph of $G$ having vertices $\tilde{v_1}, \dots, \tilde{v_t}$ corresponding to the blocks $B_1, \dots, B_t$ of $G$. Now $\widetilde{G}$ has unique canonical ordering of vertices up to reversal. Consider any of these orderings. When blocks of $G$ are arranged according to this order in its augmented adjacency matrix $A(G)$, we have same rows (and hence same columns) for vertices in the same block. Thus permutation of vertices within the same block does not change the matrix. Let $b_i=|B_i|$ for each $i=1,2,\ldots,t$. Then considering the stair sequences of $G$ and $\widetilde{G}$, it is clear that $\s_G$ is obtained uniquely from $\s_{\widetilde{G}}$ by replacing each $i$ by the subsequence 
$$(\sum\limits_{j=1}^{i-1} b_j+1,\sum\limits_{j=1}^{i-1} b_j+2,\ldots,\sum\limits_{j=1}^{i-1} b_j+b_i)$$
irrespective of the permutations of vertices within a block.
\end{proof}

\begin{rem}\label{nonreduced}
Let $G$ be a connected proper interval graph which is not reduced and $\widetilde{G}$ be the reduced graph of $G$. Then the graph $\widetilde{G}$ has a unique (up to reversal) canonical ordering of vertices, say, $b_1, \dots, b_t$ (corresponding to the blocks $B_1, \dots, B_t$) as it is connected and reduced. Now the canonical orderings of the vertices of $G$ are obtained from this ordering (and its reversal) by all possible permutation of the vertices of $G$ within each block. In all such cases $\s_G$ will remain same up to reversal.
\end{rem}

\noindent
We note that Corollary \ref{corunique} is analogous to Lemma 1 of \cite{SYKU} or Corollary 2.5 of \cite{DHH}. Any of these results enables one to calculate that the number of non-isomorphic connected proper interval graphs with $n+1$ vertices is $\displaystyle{\frac{1}{2}\left[C(n)+\binom{n}{\lfloor n/2 \rfloor}\right]}$, where $\displaystyle{C(n)=\frac{1}{n+1}\binom{2n}{n}}$ is the Catalan number. It is mentioned in \cite{SYKU} that the result was communicated personally by A. Karttunen in 2002.

\section{Structure of PTPIG}\label{s:structure}

Let us recall the definition of a proper tagged probe interval graph.

\begin{defn}\label{ptpigdef}
A tagged probe interval graph $G=(P,N,E)$ is a {\em proper tagged probe interval graph} (briefly PTPIG) if $G$ has a TPIG representation $\Set{I_x}{x\in P\cup N}$ such that $\Set{I_p}{p\in P}$ is a proper interval representation of $G_P$. We call such an interval representation a {\em PTPIG representation} of $G$. 
\end{defn}

\noindent
It is interesting to note that there are examples of $TPIG$, $G$ for which $G_P$ is a proper interval graph but $G$ is not a $PTPIG$. For example, the graph $G_b$ (see Figure \ref{fig33}) in \cite{LP} is a TPIG in which $(G_b)_P$ consists of a path of length $4$ along with $2$ isolated vertices which is a proper interval graph. But $G_b$ has no TPIG representation with a proper interval representation of $(G_b)_P$. 

\begin{figure}[t]
\begin{center}
\includegraphics[scale=1]{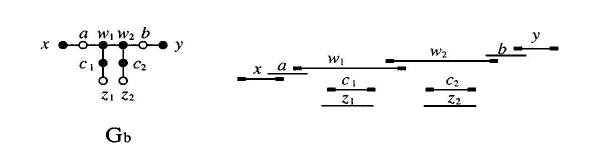}
\caption{The graph $G_b$ and its TPIG representation \cite{LP}}\label{fig33}
\end{center}
\end{figure}

\vspace{1em}\noindent
Now let us consider a graph $G=(V,E)$, in general, with an independent set $N$ and $P=V\smallsetminus N$ such that the subgraph $G_P$ of $G$ induced by $P$ is a proper interval graph. Let us order the vertices of $P$ in a canonical ordering. Now the adjacency matrix of $G$ looks like the following: 

{\tiny $$\begin{array}{cc|ccc|ccc|c}
\multicolumn{3}{c}{} & P & \multicolumn{2}{c}{} &N& \multicolumn{2}{c}{}\\ \cline{3-8}
 &&&  &&&  &&\\ 
P &&& A(P) &&& A
 (P,N) &&\\ 
 &&&  &&&  &&\\ \cline{3-8}
 &&&  &&&  &&\\ 
N &&& A(P,N)^T &&& \mathbf{0} &&\\
 &&&  &&&  &&\\ \cline{3-8}
\end{array}$$}

\noindent
Note that the (augmented) adjacency matrix $A(P)$ of $G_P$ satisfies the consecutive $1$'s property and the $P\times N$ submatrix $A(P,N)$ of the adjacency matrix of $G$ represents edges between probe vertices and nonprobe vertices. In the following lemma we obtain a necessary condition for a PTPIG.

\begin{lem}\label{lem1}
Let $G=(P,N,E)$ be a PTPIG. Then for any canonical ordering of the vertices belonging to $P$ each column of $A(P,N)$ cannot have more than two consecutive stretches of $1$'s. 
\end{lem}

\begin{proof}
Let us prove by contradiction.  Consider a canonical ordering of vertices belonging to $P$, say, $\set{u_1,u_2,\ldots, u_m}$.  Let $w_j$ be a vertex in $N$ such that in the matrix $A(P,N)$ the column corresponding to $w_j$ has at least three consecutive stretches of $1$'s. That is, there are five vertices in $P$, say $u_{i_1}, u_{i_2}, u_{i_3}, u_{i_4}$ and $u_{i_5}$ (with $i_1,i_2,i_3, i_4, i_5\in\set{1,2,\ldots, m}$) such that $i_1<i_2<i_3<i_4<i_5$ and  $u_{i_1}, u_{i_3}$ and $u_{i_5}$ are neighbors of $w_j$  while $u_{i_2}$ and $u_{i_4}$ are not neighbors of $w_j$.  Now let us prove its impossibility. We prove it case by case. 

\vspace{1em}\noindent
Let the interval corresponding to the vertex $v_{i_k}$ be $I_{v_{i_k}}=[\ell_{k}, r_{k}]$ for $k=1,2,3,4,5$ in a PTPIG representation. Now by Theorem \ref{troberts}, we have $\ell_1< \ell_2 < \ell_3 < \ell_4 < \ell_5$ and $r_1<r_2<r_3<r_4 < r_5$. Since $G$ is a PTPIG, either $\ell_i\in I_{w_j}$ or $r_i\in I_{w_j}$ for each $j=1,3,5$. 

\begin{enumerate}
\item[Case 1] \textbf{($\ell_1,\ell_5\in I_{w_j}$):} In this case, for $t$ such that $i_1 \leq t \leq i_5$, we have $\ell_t\in I_{w_j}$. In particular we have $\ell_{2}$ and  $\ell_{4}$ are in $I_{w_j}$, i.e., $u_{i_2}$ and $u_{i_4}$ are neighbors of $w_j$ which is a contradiction.   
\item[Case 2] \textbf{($r_1,r_5\in I_{w_j}$):} In this case, for all  $t$ such that $i_1 \leq t \leq i_5$, we have $r_t\in I_{w_j}$. And again here we have a contradiction just like the previous case.
\item[Case 3] \textbf{($\ell_1,r_5\in I_{w_j}$ but $r_1,\ell_5\notin I_{w_j}$):} Let $I_{w_j}$ be $[\ell_{w_j}, r_{w_j}]$. So in this case, $a_5 < \ell_{w_j}\leqslant a_1$ which is a contradiction. 
\item [Case 4] \textbf{($r_1,\ell_5\in I_{w_j}$):}  If $\ell_3\in I_{w_j}$, then $\ell_t\in I_{w_j}$ for all $t \in  \{i_3,\ldots, i_5\}$ and this would mean that $\ell_4\in I_{w_j}$ which is a contradiction. Similarly,  if $r_3\in I_{w_j}$, then $r_t\in I_{w_j}$ for all $t\in\{i_1,\ldots, i_3\}$ and then $r_2\in I_{w_j}$ which also gives a contradiction. 
\end{enumerate}
Note that all the cases are taken care of and thus each column of $A(P,N)$ cannot have more than two consecutive stretches of $1$'s. 
\end{proof}

\noindent
Unfortunately the condition in the above lemma is not sufficient as the following example shows. For convenience, we say an interval $I_p=[a,b]$ {\em contains strongly} an interval $I_n=[c,d]$ if $a<c\leqslant d<b$, where $p\in P$ and $n\in N$.\footnote{In \cite{LP2}, Sheng et al. used the term ``contains properly'' in this case. Here we consider a different term in order to avoid confusion with the definition of proper interval graph. Note that if $a\leqslant c\leqslant d<b$ or $a<c\leqslant d\leqslant b$, then also $I_p$ contains $I_n$ properly, but not strongly.}

\begin{exmp}\label{ex1}
Consider the graph $G=(V,E)$ with an independent set $N=\set{n_1,n_2}$ and $P=V \smallsetminus N =\set{p_1,p_2,\ldots ,p_6}$, where the matrices $A(P)$ and $A(P,N)$ are given in Table \ref{tapex1}. We note that $A(P)$ satisfies consecutive $1$'s property. So $G_P$ is a proper interval graph by Theorem \ref{troberts}. Also note that $G_P$ is connected and reduced. Thus the given ordering and its reversal are the only canonical ordering of vertices of $G_P$ by Lemma \ref{obsseq}. Suppose $G$ is a PTPIG with an interval representation $\Set{I_x}{x\in V}$, where $I_{p_i}=[a_i,b_i]$ and $I_{n_1}=[c,d]$ such that $a_1< a_2<\cdots <a_6$ and $b_1< b_2<\cdots <b_6$. 

\vspace{1em}\noindent
Since $p_2n_1=1$, we have either $a_2\in [c,d]$ or $b_2\in [c,d]$. Let $a_2\in [c,d]$. Now $a_4\in [c,d]$ implies $a_3\in [c,d]$. But $p_3n_1=0$, which is a contradiction. Thus $b_4\in [c,d]$ and $a_4\notin [c,d]$. But then we have $a_4<c\leqslant a_2$ as $a_4<b_4$, which is again a contradiction. Again $b_2,b_4\in [c,d]$ would imply $b_3\in [c,d]$ and consequently $p_3n_1=1$. Thus we must have $b_2,a_4\in [c,d]$ and $[a_3,b_3]$ contains $[c,d]$ strongly as $a_2<a_3<a_4$ and $b_2<b_3<b_4$. Again $p_5n_1=1$ implies $[c,d]\cap\set{a_5,b_5}\neq\emptyset$. But $p_3p_5=0$ which implies $[a_3,b_3]\cap [a_5,b_5]=\emptyset$. Thus $a_3<b_3 < a_5 < b_5$ as $a_3 <a_5$. But then $[c,d]\cap\set{a_5,b_5} = \emptyset$ (as $[c,d]\subsetneqq [a_3,b_3]$) which is a contradiction. Similar contradiction would arise if one considers the reverse ordering of vertices of $G_P$. Therefore $G$ is not a PTPIG, though each column of $A(P,N)$ does not have more than two consecutive stretches of $1$'s.
\end{exmp}

\begin{table}
{\footnotesize
$$\begin{array}{c|cccccc|}
\multicolumn{1}{c}{} & p_1 & p_2 & p_3 & p_4 & p_5 & \multicolumn{1}{c}{p_6} \\
\cline{2-7}
p_1 & 1 & 1 & 0 & 0 & 0 & 0 \\
p_2 & 1 & 1 & 1 & 1 & 0 & 0 \\
p_3 & 0 & 1 & 1 & 1 & 0 & 0 \\
p_4 & 0 & 1 & 1 & 1 & 1 & 0 \\
p_5 & 0 & 0 & 0 & 1 & 1 & 1 \\
p_6 & 0 & 0 & 0 & 0 & 1 & 1 \\
\cline{2-7}
\end{array}\qquad \begin{array}{c|cc|}
\multicolumn{1}{c}{} & n_1 & \multicolumn{1}{c}{n_2} \\
\cline{2-3}
p_1 & 0 & 0 \\
p_2 & 1 & 0 \\
p_3 & 0 & 0 \\
p_4 & 1 & 1 \\
p_5 & 1 & 1 \\
p_6 & 0 & 1 \\
\cline{2-3}
\end{array}$$}
\caption{The matrix $A(P)$ (left) and $A(P,N)$ (right) of the graph $G$ in Example \ref{ex1}}\label{tapex1}\label{ex1tab}
\end{table}

\noindent
The following is a characterization theorem for a PTPIG. For convenience, henceforth, a continuous stretch (subsequence) in a canonical sequence will be called a {\em substring}.

\begin{thm}\label{main}
Let $G=(V,E)$ be a graph with an independent set $N$ and $P=V \smallsetminus N$ such that $G_P$, the subgraph induced by $P$ is a proper interval graph. Then $G$ is a proper tagged probe interval graph with probes $P$ and nonprobes $N$ if and only if there is a canonical ordering of vertices belonging to $P$ such that the following condition holds:
\begin{enumerate}
\item[{\em {\bf (A)}}] for every nonprobe vertex $w\in N$, there is a substring in the canonical sequence with respect to the canonical ordering such that all the vertices in the substring are neighbors of $w$ and all the neighbors of $w$ are present at least once in the substring.
\end{enumerate}\label{condn34}
\end{thm}

\begin{proof} {\em Necessary condition:}\ Let $G=(V,E)$ be a PTPIG with probes $P$ and nonprobes $N$ such that $V=P\cup N$. Let $\Set{I_x=[\ell_x,r_x]}{x\in V}$ be a PTPIG representation of $G$ such that $\Set{I_u}{u\in P}$ be a proper interval representation of $G_P$. Then a probe vertex $u\in P$ is adjacent to $w\in N$ if and only if $\ell_u\in I_w$ or $r_u\in I_w$. Let $u_1,u_2,\ldots,u_p$ be a canonical ordering of vertices in $P$ that satisfies the conditions of Theorem \ref{troberts}. Now consider the corresponding canonical sequence $\s_{G_P}$, i.e., the combined increasing sequence of $\ell_{u_i}$ and $r_{u_i}$ for $i=1,2,\ldots,p$. Since the sequence is increasing and $I_w$ is an interval all the $\ell_{u_i}$'s and $r_{u_i}$'s which are belonging to $I_w$ occur consecutively in that sequence. Thus for any $w\in N$ there exists a substring of $\s_{G_P}$ such that all the vertices in the substring are neighbors of $w$ and all the neighbors of $w$ are present at least once in the substring. 

\vspace{1em}\noindent
{\em Sufficiency condition:}\ Let $G=(V,E)$ be a graph with an independent set $N$ and $P=V \smallsetminus N$ such that $G_P$, the subgraph induced by $P$ is a proper interval graph, $P=\set{u_1,u_2,\ldots,u_p}$ and $N=\set{w_1,w_2,\ldots,w_q}$. Suppose there is a canonical ordering $u_1,u_2,\ldots,u_p$ of vertices belonging to $P$ such that for any nonprobe veretex $w\in N$, there is a substring in the canonical sequence $S=\s_{G_P}$ with respect to this canonical ordering such that all the vertices in the substring are neighbors of $w$ and all the neighbors of $w$ are present at least once in the substring. Let us count the positions of each element in $S$ from $1$ to $2p$. Now for each probe vertex $u_i$, we assign the closed interval $[\ell_{u_i}, r_{u_i}]$ such that $\ell_{u_i}$ and $r_{u_i}$ are position numbers of first and second occurrences of $i$ in $S$ respectively. By definition of a canonical sequence, we have $\ell_{u_1}<\ell_{u_2}<\cdots <\ell_{u_p}$ and $r_{u_1}<r_{u_2}<\cdots <r_{u_p}$. Also since all position numbers are distinct, $\ell_{u_i}\neq r_{u_j}$ for all $i,j\in\set{1,2,\ldots,p}$. Thus this interval representation obeys the given canonical ordering of vertices belonging to $P$ and by construction the canonical sequence with respect to it is same as $S$. 

\vspace{1em}\noindent
We show that this interval representation is indeed an interval representation of $G_P$ which is proper. Let $i<j$, $i,j\in\set{1,2,\ldots,p}$. Then $\ell_{u_i}<\ell_{u_j}$ and $r_{u_i}<r_{u_j}$. Thus none of $[\ell_{u_i}, r_{u_i}]$ and $[\ell_{u_j}, r_{u_j}]$ contains other properly. Now $u_i$ is adjacent to $u_j$ in $G_P$ if and only if $u_iu_j=1$ in $A(P)$ when vertices of $A(P)$ are arranged as in the given canonical ordering. Again $u_iu_j=1$ with $i<j$ if and only if $j$ is lying between two occurrences of $i$ in the stair sequence of $A(P)$ and hence in $S$ by Proposition \ref{staircanseq}. Also since $i<j$, the second occurrence of $j$ is always after the second occurrence of $i$ in $S$. Thus $u_iu_j=1$ with $i<j$ if and only if $\ell_{u_j}\in [\ell_{u_i},r_{u_i}]$. This completes the verification that $\Set{[\ell_{u_i}, r_{u_i}]}{i=1,2,\ldots,p}$ is a proper interval representation of $G_P$ and that corresponds to $S$. 

\vspace{1em}\noindent
Next for each $j=1,2,\ldots,q$, consider the substring in the canonical sequence $S$ such that all the vertices in the substring are neighbors of $w_j$ and all the neighbors of $w_j$ are present at least once in the substring. Let the substring start at $\ell_{w_j}$ and end at $r_{w_j}$ in $S$. Then we assign the interval $[\ell_{w_j}, r_{w_j}]$ to the vertex $w_j$. If $w_j$ is an isolated vertex, then we assign a closed interval whose endpoints are greater than $\ell_{u_i}$ and $r_{u_i}$ for all $i=1,2,\ldots,p$. Now all we need to show is that $\Set{[\ell_{u_i}, r_{u_i}]}{i=1,2,\ldots,p}\cup \Set{[\ell_{w_j}, r_{w_j}]}{j=1,2,\ldots,q}$ is a PTPIG representation of $G$, i.e., if $u_i$ is a probe vertex and $w_j$ is a nonprobe vertex then there is an edge between them if and only if either $\ell_{u_i} \in [\ell_{w_j}, r_{w_j}]$ or $r_{u_i} \in [\ell_{w_j}, r_{w_j}]$. 

\vspace{1em}\noindent
First let us assume that there is an edge between $u_i$ and $w_j$. So the vertex $u_i$ must be present in the substring of $S$ that contains all the neighbors of $w_j$ and contains only the neighbors of $w_j$. Since $\ell_{w_j}$ and $r_{w_j}$ are the beginning and ending positions of the substring respectively, either $\ell_{u_i}$ or  $r_{u_i}$ must be in the interval $[\ell_{w_j}, r_{w_j}]$. Conversely, let either $\ell_{u_i} \in [\ell_{w_j}, r_{w_j}]$ or $r_{u_i} \in [\ell_{w_j}, r_{w_j}]$.  Then we have either $\ell_{u_i}$ or $r_{u_i}$ must be present in the substring. Since the substring contains vertices that are neighbors of $w_j$, we have $u_i$ must be a neighbor of $w_j$. 
\end{proof}

\begin{rem}\label{remuniq}
If $G$ is a PTPIG such that $G_P$ is connected and reduced, then there is a unique (up to reversal) canonical ordering of vertices belonging to $P$, as we mentioned at the beginning of Section~\ref{sproper}. Thus the corresponding canonical sequence is also unique up to reversal. Also if condition (A) holds for a canonical sequence, it also holds for its reversal. Thus in this case condition (A) holds for {\bf any} canonical ordering of vertices belonging to $P$.
\end{rem}

\noindent
\begin{table}[hb]
\begin{center}
{\footnotesize
$\begin{array}{c|cccccccc|}
\multicolumn{1}{c}{} & p_1 & p_2 & p_3 & p_4 & p_5 & p_6 & p_7 & \multicolumn{1}{c}{p_8} \\
\cline{2-9}
 p_1 & 1 & 1 & 0 & 0 & 0 & 0 & 0 & 0 \\
 p_2 & 1 & 1 & 1 & 1 & 1 & 0 & 0 & 0 \\
 p_3 & 0 & 1 & 1 & 1 & 1 & 1 & 0 & 0 \\
 p_4 & 0 & 1 & 1 & 1 & 1 & 1 & 0 & 0 \\
 p_5 & 0 & 1 & 1 & 1 & 1 & 1 & 1 & 0 \\
 p_6 & 0 & 0 & 1 & 1 & 1 & 1 & 1 & 0 \\
 p_7 & 0 & 0 & 0 & 0 & 1 & 1 & 1 & 1 \\
 p_8 & 0 & 0 & 0 & 0 & 0 & 0 & 1 & 1 \\
\cline{2-9}
\end{array}$\qquad
$\begin{array}{c|cccccc|}
\multicolumn{1}{c}{} & n_1 & n_2 & n_3 & n_4 & n_5 & \multicolumn{1}{c}{n_6} \\
\cline{2-7}
p_1 & 0 & 0 & 0 & 1 & 0 & 1 \\
p_2 & 1 & 1 & 0 & 1 & 0 & 1 \\
p_3 & 0 & 1 & 0 & 1 & 0 & 1 \\
p_4 & 0 & 1 & 1 & 1 & 0 & 1 \\
p_5 & 1 & 1 & 1 & 1 & 0 & 1 \\
p_6 & 1 & 1 & 1 & 1 & 0 & 1 \\
p_7 & 0 & 0 & 1 & 0 & 0 & 1 \\
p_8 & 0 & 0 & 1 & 0 & 0 & 1 \\
\cline{2-7}
\end{array}$}
\caption{The matrix $A(P)$ and $A(P,N)$ of the graph $G$ in Example \ref{ex2}}\label{tap}
\end{center}
\end{table}

\begin{table}[hb]
{\tiny
$$\begin{array}{cc|cccccccc|cccccc|}
& \multicolumn{1}{c}{} & [1,3] & [2,7] & [4,9] & [5,10] & [6,12] & [8,13] & [11,15] & \multicolumn{1}{c}{[14,16]} & [6,8] & [4,10] & [10,16] & [1,10] & [17,17] & \multicolumn{1}{c}{[1,16]} \\
& \multicolumn{1}{c}{} & p_1 & p_2 & p_3 & p_4 & p_5 & p_6 & p_7 & \multicolumn{1}{c}{p_8} & n_1 & n_2 & n_3 & n_4 & n_5 & \multicolumn{1}{c}{n_6} \\
\cline{3-16}
\text{$[1,3]$} & p_1 & 1 & 1 & 0 & 0 & 0 & 0 & 0 & 0 & 0 & 0 & 0 & 1 & 0 & 1 \\
\text{$[2,7]$} & p_2 & 1 & 1 & 1 & 1 & 1 & 0 & 0 & 0 & 1 & 1 & 0 & 1 & 0 & 1 \\
\text{$[4,9]$} & p_3 & 0 & 1 & 1 & 1 & 1 & 1 & 0 & 0 & 0 & 1 & 0 & 1 & 0 & 1 \\
\text{$[5,10]$} & p_4 & 0 & 1 & 1 & 1 & 1 & 1 & 0 & 0 & 0 & 1 & 1 & 1 & 0 & 1\\
\text{$[6,12]$} & p_5 & 0 & 1 & 1 & 1 & 1 & 1 & 1 & 0 & 1 & 1 & 1 & 1 & 0 & 1 \\
\text{$[8,13]$} & p_6 & 0 & 0 & 1 & 1 & 1 & 1 & 1 & 0 & 1 & 1 & 1 & 1 & 0 & 1 \\
\text{$[11,15]$} & p_7 & 0 & 0 & 0 & 0 & 1 & 1 & 1 & 1 & 0 & 0 & 1 & 0 & 0 & 1 \\
\text{$[14,16]$} & p_8 & 0 & 0 & 0 & 0 & 0 & 0 & 1 & 1 & 0 & 0 & 1 & 0 & 0 & 1 \\
\cline{3-16}
\end{array}$$}
\caption{A proper tagged probe interval representation of the graph $G$ in Example \ref{ex2}}\label{tnew}
\end{table}

Let us illustrate the above theorem by the following example.

\begin{exmp}\label{ex2}
Consider the graph $G=(V,E)$ with an independent set $N=\set{n_1,n_2,\ldots ,n_6}$ and $P=V \smallsetminus N =\set{p_1,p_2,\ldots ,p_8}$, where the matrices $A(P)$ and $A(P,N)$ are given in Table \ref{tap}. First note that $A(P)$ satisfies consecutive $1$'s property. So $G_P$ is a proper interval graph. Secondly, each column of $A(P,N)$ does not have more than two consecutive stretches of $1$'s. Now the canonical sequence $S=\s_{G_P}=(1\ 2\ 1\ 3\ 4\ 5\ 2\ 6\ 3\ 4\ 7\ 5\ 6\ 8\ 7\ 8)$. The required substrings of probe neighbors for nonprobe vertices $n_1,n_2,\ldots ,n_6$ are $(5\ 2\ 6)$, $(3\ 4\ 5\ 2\ 6\ 3\ 4)$, $(4\ 7\ 5\ 6\ 8\ 7\ 8)$, $(1\ 2\ 1\ 3\ 4\ 5\ 2\ 6\ 3\ 4)$, $\emptyset$, $S$ respectively. Note that $G$ is indeed a PTPIG with an interval representation shown in Table \ref{tnew} which is constructed by the method described in the sufficiency part of Theorem \ref{main}.

\vspace{1em}\noindent
Now consider the graph $G$ in Example \ref{ex1} which is not a PTPIG. From Table \ref{ex1tab} we compute $S=S(G_P)=(1\ 2\ 1\ 3\ 4\ 2\ 3\ 5\ 4\ 6\ 5\ 6)$. The graph $G_P$ is connected and reduced. So $S$ is unique up to reversal. Note that the nonprobe vertex $n_1$ is adjacent to probe vertices $\set{p_2,p_4,p_5}$ and there is no substring in $S$ containing only $\set{2,4,5}$.
\end{exmp}

\begin{defn}
Let $G=(V,E)$ be a graph with an independent set $N$ and $P=V \smallsetminus N$ such that $G_P$, the subgraph induced by $P$ is a proper interval graph. Let $\s_{G_P}$ be a canonical sequence of $G_P$.  Let $w\in N$. If there exists a substring in $\s_{G_P}$ which contains all the neighbors of $w$ and all the vertices in the substring are neighbors of $w$ then we call the substring a {\em perfect substring of $w$ in $G$}. If the canonical sequence $\s_{G_P}$ contains a {\em perfect substring of $w$ in $\s_{G_P}$} for all $w\in N$, we call it a {\em perfect canonical sequence for $G$}.
\end{defn}

\begin{prop}\label{maincor}
Let $G=(P, N, E)$ be a PTPIG such that $G_P$ is a connected reduced proper interval graph and $\s_{G_P}$ be a canonical sequence of $G_P$. Then for any nonprobe vertex $w\in N$, there cannot exist more than one disjoint perfect substring of $w$ in $\s_{G_P}$, unless the substring consists of a single element. 
\end{prop}

\begin{proof}
Let $u_1,u_2,\ldots,u_p$ be the canonical ordering of the probe vertices of $G$ with the proper interval representation $\Set{[\ell_i,r_i]}{i=1,2,\ldots,p}$ that satisfies the condition~\ref{order} of Theorem~\ref{troberts} and $S$ be the corresponding canonical sequence $\s_{G_P}$. We first note that, since each vertex in $S$ appears twice, there cannot be more than two disjoint perfect substrings of $S$.

\vspace{1em}\noindent
Now suppose there is a nonprobe vertex, say, $w$ in $G$ such that there are two disjoint perfect substrings of length greater than $1$. We will refer to these substrings as the first substring and second substring corresponding to the relative location of the substrings in $S$. Now in $S$, each number $i$ appears twice due to $l_i$ and $r_i$ only. Thus if we think of the canonical sequence as an ordering of $\ell_i$'s and $r_i$'s, then we have that the first substring contains all the $\ell_i$'s and the second substring contains all the $r_i$'s for all the probe vertices $u_i$ those are neighbors of $w$, as $\ell_i<r_i$ and both substrings contain all numbers $i$ such that $u_i$ is a neighbor of $w$. 

\vspace{1em}\noindent
Moreover due to the increasing order of $\ell_i$'s and $r_i$'s, both the substrings contain numbers $k,k+1,\ldots,k+r$ for some integers $k,r$ with $1\leqslant k\leqslant m$ and $1\leqslant r\leqslant m-k$. So the first substring must comprise of some consecutive collection of $\ell_i$ and similarly for the second substring, i.e., the first substring is $\ell_k, \ell_{k+1}, \dots, \ell_{k+r}$ and the second substring is $r_k, r_{k+1}, \dots, r_{k+r}$ (in $\is_{G_P}$). Therefore the vertices $u_k, \dots, u_{k+r}$ form a clique. 

\vspace{1em}\noindent
Now suppose $u_i$ is adjacent to $u_{k+t}$ for some $i<k$ and $1\leq t\leq r$. Then $\ell_i<\ell_k$ and $\ell_{k+r}<r_i$ as $\ell_k$ to $\ell_{k+r}$ are consecutive in the first substring (in $\is_{G_P}$). But this implies $u_i$ is adjacent to all $u_k,u_{k+1},\ldots,u_{k+r}$. Similarly, one can show that if $u_j$ is adjacent to $u_{k+t}$ for some $j>k+r$ and $1\leq t\leq r$. Then $u_j$ is adjacent to all $u_k,u_{k+1},\ldots,u_{k+r}$. Thus (closed) neighbors of $u_k,u_{k+1},\ldots,u_{k+r}$ are same in $G_P$ which contradicts the assumption that $G_P$ is reduced as $r\geq 1$.
\end{proof}

\noindent
In fact, we can go one step more in understanding the structure of a PTPIG. If $G$ is a PTPIG, not only there cannot be two disjoint perfect substrings (of length more than $1$) for any nonprobe vertex in any canonical sequence but also any two perfect substrings for the same vertex must intersect at at least two places, except two trivial cases.  

\begin{lem}\label{lemtcp}
Let $G=(P,N,E)$ be a PTPIG such that $G_P$ is a connected reduced proper interval graph with a canonical ordering of vertices $\set{u_1,u_2,\ldots,u_p}$ and let $\vs_{G_P}$ be the corresponding vertex canonical sequence of $G_P$. Let $w\in N$ be such that $w$ has at least two neighbors in $P$ and $T_1$ and $T_2$ be two perfect substrings for $w$ in $\vs_{G_P}$ intersecting in exactly one place. Then one of the following holds:
\begin{enumerate}
\item $\vs_{G_P}$ begins with $u_1u_2u_1$ and only $u_1$ and $u_2$ are neighbors of $w$.
\item $\vs_{G_P}$ ends with $u_pu_{p-1}u_p$ and only $u_{p-1}$ and $u_p$ are neighbors of $w$.
\end{enumerate}
\end{lem}

\begin{proof}
Let $[a_i,b_i]$ be the interval corresponding to $u_i$ for $i=1,2,\ldots,p$. Let the place where $T_1$ and $T_2$ intersect be the first occurrence of the vertex $u_{k}$.

\vspace{1em}\noindent
Without loss of generality, let the substring $T_1$ end with the first occurrence of $u_k$ and the substring $T_2$ start with the first occurrence of $u_{k}$. Thus for all $i>k$, the vertex $u_{i}$ cannot appear before the first occurrence of $u_{k}$ in the $\vs_{G_P}$. So $T_1$ does not contain any $u_{i}$ such that $i>k$. Thus $w$ is not a neighbor of any $u_{i}$ such that $i> k$.  Note that, it also means that for any vertex in the neighbor of $w$ (except for $u_{k}$) the substring $T_1$ contains the first occurrence, while the substring $T_2$ contains the second occurrence. Thus the vertices in the neighborhood of $w$ has to be consecutive vertices in the canonical ordering of $G_P$. Let the vertices in the neighborhood of $w$ be $u_{k-r}, \dots, u_{k-1}, u_{k}$, where $1\leq r\leq k-1$. 

\vspace{0.5em}\noindent
Now for any vertex $u_i$ such that $i< k-r$, we have $u_i$ is not in $T_1$ and $T_2$. So the first occurrence of $u_i$ is before the first occurrence of $u_{k-r}$ and the second occurrence of $u_i$ is either also before the first occurrence of $u_{k-r}$ or after $T_2$, i.e., after the second occurrence of $u_{k-r}$. But if the second case happens, then we would violate the fact that $G_P$ is proper.  So the only option is that both the first and second occurrence of $u_i$ is before the first occurrence of $u_{k-r}$  and this would violate the condition that the graph $G_P$ is connected. So the only option is there exists no $u_i$ such that $i< k-r$. So $k-r =1$.  Thus we have the neighbors of $w$ precisely $u_1, \dots, u_k$. 

\vspace{0.5em}\noindent
Now if we look at the interval canonical sequence of $G_P$, we have $T_1$ corresponds to $a_1, \dots, a_k$ and $T_2$ corresponds to $a_k, b_1, \dots, b_{k-1}$.  But this would mean that all the vertices $u_{1}, \dots, u_{k-1}$ have the same (closed) neighborhood in $G_P$ which is not possible as we assumed $G_P$ is reduced, unless the set $\{u_{1}, \dots, u_{k-1}\}$ is a single element set. In that case, $w$ has neighbors $u_1$ and $u_2$ and the $T_1$ and $T_2$ correspond to $a_1, a_2$ and $a_2, b_1$ respectively (in $\is_{G_P}$). This is the first option in the Lemma. By similar argument, if we assume that $T_1$ and $T_2$ intersect in the second occurrence of the vertex $u_{k}$, we get the other option. 
\end{proof}

\noindent
Now let us consider what happens when the graph $G_P$ induced by the probe vertices is not necessarily reduced. Let the blocks of $G_P$ be $B_1, \dots, B_t$. Then a canonical ordering of vertices in the reduced graph $\widetilde{G_P}$ can be considered as a {\em canonical ordering of blocks} of $G_P$ and the corresponding canonical sequence $\s_{\widetilde{G_P}}$ can be considered as a {\em canonical sequence of blocks} of $G_P$.  Then a  substring in $\s_{\widetilde{G_P}}$ can also be called as a {\em block substring}. If there is a vertex in a block $B_j$ that is a neighbor of a nonprobe vertex $w$, we call $B_j$ a {\em block-neighbor} of $w$. Also if a block substring contains all the block neighbors of $w$ and all the blocks in it are block-neighbors of $w$, then we call the block substring a {\em perfect block substring of $w$ in $\s_{\widetilde{G_P}}$}. Then the following corollaries are immediate from Proposition~\ref{maincor} and Lemma~\ref{lemtcp}.

\begin{cor}\label{cor:disjoint}
Let $G=(P,N,E)$ be a PTPIG such that $G_P$ is a connected proper interval graph. Let $\s_{\widetilde{G_P}}$ be a canonical sequence for the reduced graph $\widetilde{G_P}$. Then for any $w\in N$, there cannot exist more than one disjoint perfect block substrings of $w$ in $\s_{\widetilde{G_P}}$, unless the block substring consists of a single element.
\end{cor}

\begin{cor}\label{cor:1intersection}
Let $G=(P,N,E)$ be a PTPIG such that $G_P$ is a connected proper interval graph with a canonical ordering of blocks $\set{B_1,B_2,\ldots,B_t}$ and let $\vs_{\widetilde{G_P}}$ be the corresponding vertex canonical sequence of blocks of $G_P$. Let $w\in N$ be such that $w$ has at least two block-neighbors in $G_P$ and $T_1$ and $T_2$ be two perfect block substrings for $w$ in $\vs_{\widetilde{G_P}}$ intersecting in exactly one place. Then one of the following holds:
\begin{enumerate}
\item $\vs_{\widetilde{G_P}}$ begins with $B_1B_2B_1$ and only $B_1$ and $B_2$ are block-neighbors of $w$.
\item $\vs_{\widetilde{G_P}}$ ends with $B_tB_{t-1}B_t$ and only $B_{t-1}$ and $B_t$ are block-neighbors of $w$.
\end{enumerate}
\end{cor}

\section{Recognition algorithm}\label{reco}

In this section, we present a linear time recognition algorithm for PTPIG. That is, given a graph $G=(V,E)$, and a partition of the vertex set into $N$ and $P = V\smallsetminus N$ we can check, in time $O(|V|+|E|)$, if the graph $G(P,N,E)$ is a PTPIG. Now $G=(P,N,E)$ is a PTPIG if and only if it is a TPIG (i.e., it satisfies the three conditions in Definition~\ref{defn1}) and $G_P$ is a proper interval graph for a TPIG representation of $G$. Note that it is easy to check in linear time if the graph $G$ satisfies the first condition, namely, if $N$ is an independent set in the graph. Now for testing if the graph satisfies the other two properties we will use the characterization we obtained in Theorem~\ref{main}. 

\vspace{0.5em}\noindent
We will use the recognition algorithm for proper interval graph $H=(V^\prime,E^\prime)$ given by Booth and Lueker \cite{BL76} as a black box that runs in $O(|V^\prime|+|E^\prime|)$. The main idea of their algorithm is that $H$ is a proper interval graph if and only if the adjacency matrix of the graph satisfies the consecutive $1$'s property. In other words, $H$ is a proper interval graph if and only if there is an ordering of the vertices of $H$ such that for any vertices $v$ in $H$,  the neighbors of $v$ must be consecutive in that ordering. So for every vertex $v$ in $H$ they consider restrictions, on the ordering of the vertices, of the form ``all vertices in the neighborhood of $v$ must be consecutive". This is done using the data structure of $PQ$-trees.
The PQ-tree helps to store all the possible orderings that respect all these kind of restrictions. It is important to note that all the orderings that satisfy all the restrictions are precisely all the canonical orderings of vertices of $H$. 

\vspace{0.5em}\noindent
The main idea of our recognition algorithm is that if the graph $G=(P,N,E)$ is PTPIG then, from Condition \textbf{(A)} in Theorem~\ref{main}, we can obtain a series of restrictions on the ordering of vertices that also can be ``stored" using the PQ-tree data structure. These restrictions are on and above the restrictions that we need to ensure the graph $G_P$ is a proper interval graph. If finally there exists any ordering of the vertices that satisfy all the restrictions, then that ordering will be a canonical ordering that satisfies the condition  \textbf{(A)} in Theorem~\ref{main}.  So the main challenge is to identify all the extra restrictions on the ordering and how to store them in the PQ-tree. 

\vspace{0.5em}\noindent
Once we have verified that the graph $G=(P,N,E)$ satisfies the first condition in Definition~\ref{ptpigdef} and we have stored all the possible canonical ordering of the vertices of the subgraph $G_P=(P,E_1)$ of $G$ induced by $P$ in a PQ-tree (in $O(|P|+|E_1|)$ time), we proceed to find the extra restrictions those are necessary to be applied on the orderings. We present our algorithm for three cases - each case handling a class of graphs that is a generalization of the class of graphs handled in the previous one. 

\begin{itemize}
\item[CASE \ \ I:] First we consider the case when $G_P$ is a connected reduced proper interval graph.
\item[CASE \ II:] Next we consider the case when $G_P$ is a connected proper interval graph, but not necessarily reduced.
\item[CASE III:] Finally we consider the general case when the graph $G_P$ is a proper interval graph, but may not be connected or reduced.
\end{itemize}

\noindent
For all the cases we will assume that the vertices in $P$ are $v_1, \dots, v_p$, the vertices in $N$ are $w_1, \dots, w_q$ and $A_j$ be the adjacency list of the vertex $w_j$ and let $d_j$ be the degree of the vertex $w_j$. So the neighbors of $w_j$ are $A_j[1],A_j[2], \ldots, A_j[d_j]$ for $j=1,2,\ldots,q$.

\subsection{CASE I: The graph $G_P$ is a connected reduced proper interval graph} 

By Lemma~\ref{obsseq}, there is a unique (up to reversal) canonical ordering of the vertices of $G_P$. By Theorem~\ref{main}, we know that the graph $G$ is PTPIG if and only if the following condition is satisfied: 

\vspace{0.5em}\noindent
\textbf{Condition (A1):} For all $1\leq j \leq q$, there is a substring in $\s_{G_P}$ where only the neighbors of $w_j$ appear and  all the neighbors of $w_j$ appear at least once.

\vspace{1em}\noindent
\begin{algorithm}[H]\label{algonew11}
\SetKwData{Left}{left}\SetKwData{This}{this}\SetKwData{Up}{up}
\SetKwFunction{Union}{Union}\SetKwFunction{FindCompress}{FindCompress}
\SetKwInOut{Input}{Input}\SetKwInOut{Output}{Output}
\footnotesize{
\Input{An array $A[1, T]$ containing only $0/1$ entries, $d$ (Assume $A[T+1] = 0$).}
\Output{Output all $(s, t)$ pairs such that $t-s \geq d$ and for all $s\leq i\leq t$, $A[i] =1$, \\
 (either $s=1$ or $A[s-1] = 0$) and (either $t=T$ or $A[t+1]=0$).}
Initialize $\ell = r = 0$
\While{$r\leq T$}{
\If{$A[r+1]= 1$}{$r = r+1$}
\If{$A[r+1] = 0$}{
\If{$r-\ell \geq d$}{Output $(\ell+1, r)$}
$\ell = r+1$\\
$r= r+1$
}
}
}
\caption{Finding-1Subsequence}
\end{algorithm}
\begin{algorithm}[H]\label{algonew12}
\SetKwData{Left}{left}\SetKwData{This}{this}\SetKwData{Up}{up}
\SetKwFunction{Union}{Union}\SetKwFunction{FindCompress}{FindCompress}
\SetKwInOut{Input}{Input}\SetKwInOut{Output}{Output}
\footnotesize{
\Input{Given a graph $G=(V, E)$} 
\Output{Accept if the \textbf{Condition (A1)} is satisfied}
\For{$j\leftarrow 1$ \KwTo $q$}{
Initialize two arrays $X$ and $Y$ of length $4d_j +1$ by setting all the entries to $0$.\\
Let $\ell = L(A_j[1])$ and $r=R(A_j[1])$. \\
\For{each $k\leftarrow 1$ \KwTo $d_j$}{\label{forins}
Let $s = L(A_j[k])$ and $t = R(A_j[k])$\\
If ($\ell - 2d_j \leq s\leq \ell+2d_j$) \textbf{Mark}  $X[s-\ell + 2d_j +1] = 1$\\
If ($\ell - 2d_j \leq t\leq \ell+2d_j$) \textbf{Mark}  $X[t-\ell + 2d_j +1] = 1$\\
If ($r - 2d_j \leq s\leq r+2d_j$) \textbf{Mark}  $Y[s-r + 2d_j +1] = 1$\\
If ($r - 2d_j \leq t\leq r+2d_j$) \textbf{Mark}  $Y[t-r + 2d_j +1] = 1$\\
}
\For{all $(a, b)$ output of \textbf{Finding-1Subsequence($X,d_j$)}}{
Let $\hat{a} = a + \ell - 2d_j -1$ and $\hat{b} = b + \ell - 2d_j -1$.\\
\If{$\forall 1\leq k \leq d_j$ either $\hat{a}\leq L(A_j[k]) \leq\hat{b}$ or  $\hat{a}\leq R(A_j[k]) \leq\hat{b}$}{$j = j+1$\\ \textbf{Go To} Step 1}
}
\For{all $(a, b)$ output of \textbf{Finding-1Subsequence($Y,d_j$)}}{
Let $\hat{a} = a + r - 2d_j -1$ and $\hat{b} = b + r - 2d_j -1$.\\
\If{$\forall 1\leq k \leq d_j$ either $\hat{a}\leq L(A_j[k]) \leq\hat{b}$ or  $\hat{a}\leq R(A_j[k]) \leq\hat{b}$}{$j = j+1$ \\  \textbf{Go To} Step 1}
}
\textbf{REJECT}\\
}
\textbf{ACCEPT}
}
\caption{Testing Condition (A1)}
\end{algorithm}

\noindent
In this case, when the graph $G_P$ is connected reduced proper interval graph, since there is a unique canonical ordering of the vertices, all we have to do is to check if the corresponding canonical sequence satisfies the Condition (A1). So the rest of the algorithm in this case is to check if the property is satisfied. 

\vspace{0.5em}\noindent
{\bf Idea of the algorithm:} 
Since we know the canonical sequence $\s_{G_P}$ (or obtain by using known algorithms described before in $O(|P|+|E_1|)$ time, where $E_1$ is the set of edges between probe vertices), we can have two look up tables $L$ and $R$ such that for any vertex $v_i\in P$, the $L(v_i)$ and $R(v_i)$ has the index of the first and the second appearance of $v_i$ in $\s_{G_P}$ respectively. We can obtain the look up tables in time $O(|P|)$ steps. 

\vspace{0.5em}\noindent
Also by $\s_{G_P}[k_1, k_2]$ (where $1\leq k_1\leq k_2 \leq 2p$) we will denote the substring of the canonical sequence sequence $\s_{G_P}$ that start at the $k_1^\text{th}$ position and ends at the $k_2^\text{th}$ position in $\s_{G_P}$. 

\vspace{0.5em}\noindent
To check the Condition (A1), we will go over all the $w_j \in N$. For $j\in\{1,2,\ldots,q\}$, let $L(A_j[1]) = \ell_j$ and $R(A_j[1]) = r_j$. Now since all the neighbors of $w_j$ have to be in a substring, there must be a substring of length at least $d_j$ and at most $2d_j$ (as each number appears twice) in $\s_{G_P}[\ell_j-2d_j, \ell_j+2d_j]$ or $\s_{G_P}[r_j-2d_j, r_j+2d_j]$ which contains only and all the neighbors of $w_j$. We can identify all such possible substrings by first marking the positions in $\s_{G_P}[\ell_j-2d_j, \ell_j+2d_j]$ and $\s_{G_P}[r_j-2d_j, r_j+2d_j]$ those are neighbors of $w_j$ and then by doing a double pass (Algorithm~\ref{algonew11}), we find all the possible substrings of length greater than or equal to $d_j$ in $\s_{G_P}[\ell_j-2d_j, \ell_j+2d_j]$ and $\s_{G_P}[r_j-2d_j, r_j + 2d_j]$  that contains only neighbors of $w_j$. 

\vspace{0.5em}\noindent
We prove the correctness and run-time of the algorithm in the following theorem: 

\begin{thm}
Let $G=(V,E)$ be a graph with an independent set $N$ and $P=V\smallsetminus N$ such that $G_P$ is a connected reduced proper interval graph. Then Algorithm \ref{algonew12} correctly decides whether $G$ is a PTPIG with probes $P$ and nonprobes $N$ in time $O(|P|+|N|+|E_2|)$, where $E_2$ is the set of edges between probes $P$ and nonprobes $N$.
\end{thm}

\begin{proof}
For each $j=1,2,\ldots,q$, we test the Condition (A1). In Line 3, $\ell$ and $r$ are the first and second appearance of the first probe neighbor of $w_j$ respectively. In Lines 4--9, we generate two $(0,1)$ arrays $X$ and $Y$, each of length $4d_j+1$. The probe neighbors of $w_j$ lying between $[\ell-2d_j, \ell+2d_j]$ are marked $1$ in $X$ with positions translated by $2d_j+1-\ell$. Similarly $Y$ is formed by marking $1$ for the probe neighbors of $w_j$ lying between $[r-2d_j, r+2d_j]$ and translated by $2d_j+1-r$. Translations are required to avoid occurrence of a negative position. Then using Algorithm \ref{algonew11} we find all substrings of probe neighbors of length greater than or equal to $d_j$ in Lines 10 and 15. Finally, in Lines 12--13 and 17--18, we check that whether the substring contains all probe neighbors of $w_j$. Thus by Theorem~\ref{main} and Remark~\ref{remuniq}, Algorithm \ref{algonew12} correctly decides whether $G$ is a PTPIG with probes $P$ and nonprobes $N$ by testing that whether the Condition A1 is satisfied.

\vspace{1em}\noindent
Now we compute the running time. For each $j=1,2,\ldots,q$, Algorithm \ref{algonew11} requires $O(d_j)$ time. Also Lines 4--9, 12 and 17 of Algorithm \ref{algonew12} require $O(d_j)$ time. Other steps require constant times. Thus Algorithms \ref{algonew11} and \ref{algonew12} run in $O(|E_2| + |N|)$ times, where $E_2$ is the set of edges between probes $P$ and nonprobes $N$ and the look up tables $L$ and $R$ were obtained in $O(|P|)$ time. Thus the total running time is $O(|P|+|N|+|E_2|)$.
\end{proof}

\begin{rem}
{\small Since given a graph $G=(V,E)$ with an independent set $N$ and $P=V\smallsetminus N$, the checking whether $G_P=(P,E_1)$ is a proper interval graph and obtaining $\s_{G_P}$ (from any proper interval representation of $G_P$) requires $O(|P|+|E_1|)$ time, the total recognition time is $O(|P|+|N|+|E_1|+|E_2|)=O(|V|+|E|)$.}
\end{rem}

\subsection{CASE II: The graph $G_P$ is a connected (not necessarily reduced) proper interval graph} \label{algostep2}

In this case, that is when the graph $G_P$ is not reduced, we cannot say that there exists a unique canonical ordering of the vertices of $G_P$. So by Theorem~\ref{main}, all we can say is that among the set of canonical orderings of the vertices of $G_P$, is there an ordering such that the corresponding canonical sequence satisfies Condition (A) of Theorem~\ref{main}? As mentioned before we will assume that we have all the possible canonical ordering of the vertices of $G_P$ stored in a PQ-tree. Now we will impose more constraints on the orderings so that the required condition is satisfied. 

\vspace{1em}\noindent
Let $\widetilde{G_P}$ be the reduced graph of $G_P$. By Remark~\ref{nonreduced}, $\widetilde{G_P}$ has a unique (up to reversal) canonical ordering of vertices, say, $b_1, \dots, b_t$ (corresponding to the blocks $B_1, \dots, B_t$ of the vertices of $G_P$) and the canonical orderings of the vertices of $G_P$ are obtained by all possible permutations of the vertices of $G$ within each block. 

\vspace{1em}\noindent
Since we know the canonical sequence $\s_{\widetilde{G_P}}$ for the vertices in $\widetilde{G_P}$, we can have two look up tables $L$ and $R$ such that for any vertex $b_k\in \widetilde{G_P}$, $L(b_k)$ and $R(b_k)$ have the indices of the first and the second appearance of $b_k$ in the $\s_{\widetilde{G_P}}$
respectively. Abusing notation, for any probe vertex $v\in B_k$, by $L(v)$ and $R(v)$ we will denote $L(b_k)$  and $R(b_k)$ respectively. 
We can obtain the look up tables in time $O(|P|)$ steps. 

\begin{defn}\label{def:012}
For any $w\in N$ and any block $B_k$, we say that $B_k$ is a {\em block-neighbor} of $w$ if there exists at least one vertex in $B_k$ that is a neighbor of $w$. If all the vertices in $B_k$ are neighbors of $w$, we call $B_k$ a {\em full-block-neighbor} of $w$ and if there exists at least one vertex in $B_k$ that is a neighbor of $w$ and at least one vertex of $B_k$ that is not a neighbor of $w$ we call it a {\em partial-block-neighbor} of $w$. Also for any $w\in N$ let us define the function $f_{w}: \{1,2, \ldots, t\} \to \{0,1,2\}$ as 
$$ f_{w}(k) = \left\{ \begin{array}{ll}
         0, & \mbox{if $B_k$ is not a block-neighbor of $w$};\\
        1, & \mbox{if $B_k$ is a full-block-neighbor of $w$};\\
        2, & \mbox{if $B_k$ is a partial-block-neighbor of $w$}.\end{array} \right.$$
By abuse of notation, for any probe vertex $v\in B_k$ by $f_{w}(v)$ we will denote $f_w(k)$. 
\end{defn}

\noindent
Note that for any $w_j\in N$ one can compute the function $f_{w_j}$ in $O(d_j)$ number of steps and can store the function $f_{w_j}$ in an array of size $t$. 
 
\subsubsection{Idea of the Algorithm}\label{idea2}
If $G$ is PTPIG then from Condition (A) in Theorem~\ref{main} we can see that the following condition is a necessary (though not a sufficient) condition: 

\vspace{1em}\noindent
\textbf{Condition B1:} For all $1\leq j \leq q$, there is a substring of $\s_{\widetilde{G_P}}$ where only the block-neighbors of $w_j$ appear and all the block-neighbors of $w_j$ appear at least once and any block that is not at the beginning or end of the substring must be a full-block-neighbor of $w_j$.

\vspace{1em}\noindent
Though the above Condition B1 is not a sufficient condition, but as a first step we will check if the graph satisfies the Condition B1. For every $w_j\in N$, we will identify (using algorithm \textbf{CheckConditionB1($w_j$)} (see Algorithm \ref{algonew2})) all possible maximal substrings of $\s_{G_P}$  that can satisfy Condition B1. If such a substring exists, then CheckConditionB1($w_j$) outputs the block numbers that appear at the beginning and end of the substring. Let for some $w_j$, CheckConditionB1($w_j$) outputs $(k_1, k_2)$, then note that $1\leq k_1\leq k_2 \leq t$ and $k_1$ and $k_2$ are the only possible partial neighbors of $w_j$.  The algorithm CheckConditionB1($w_j$) is very similar to the algorithm \textbf{Testing Condition (A1)} that we described for the CASE I, when we assumed $G_P$ is a connected reduced proper interval graph.

\vspace{1em}\noindent
Now as we mentioned earlier the Condition (B1) is not sufficient for $G$ to be a PTPIG. For $G$ to be a PTPIG (that is, to satisfy Condition (A) of Theorem~\ref{main}) we will have to find a suitable canonical ordering of the vertices or in other words, by Remark~\ref{nonreduced}, we need to find suitable ordering of vertices in each block. Depending on $(k_1, k_2)$ which is the output by CheckConditionB1($w_j$), we have a number of cases and for each of the cases some restrictions will be imposed on the ordering of the vertices within blocks. 

\vspace{1em}\noindent
Let $\sigma_1, \dots, \sigma_t$ be the ordering of the vertices of blocks $B_1, B_2, \dots, B_t$ respectively. For every $w_j\in N$, let us denote by $Ngb_k(w_j)$, the vertices in the block $B_k$ those are neighbors of $w_j$. We list the different cases and restriction on $\sigma_1, \dots, \sigma_t$ that would be imposed in each of the cases when the algorithm \textbf{CheckConditionB1($w_j$)} outputs $(k_1, k_2)$.

\noindent
\begin{itemize}
\item \textbf{Category 0: There are no partial-block-neighbors of $w_j$.} In this case Condition (B1) is sufficient. 
\item \textbf{Category 1: There exists exactly one partial-block-neighbor of $w_j$.} 
Let $k$ be the block that is a partial-block-neighbor of $w_j$.  In this case we will have 4 subcategories: 
\begin{itemize} 
\item \textbf{Category 1a: $B_k$ is the only block-neighbor of $w_j$.}  In that case, $\sigma_k$ must ensure that the vertices in $Ngb_k(w_j)$ must be contiguous. 

\item  \textbf{Category 1b:  All pairs returned by CheckConditionB1($w_j$) are of the form $(k, k_2)$ with $k_2 \neq k$.} In that case,  $\sigma_k$ must ensure that the vertices in $Ngb_k(w_j)$ must be contiguous and must be flushed to the Right. 

\item  \textbf{Category 1c:  All pairs returned by CheckConditionB1($w_j$) are of the form $(k_1, k)$ with $k_1\neq k$.} In that case,  $\sigma_k$ must ensure that the vertices in $Ngb_k(w_j)$ must be contiguous and must be flushed to the Left. 

\item \textbf{Category 1d: All other cases.} In that case,  CheckConditionB1($w_j$) returns $(k, k)$ or returns both $(k , k_2)$ and $(k_1, k)$. In both the cases $\sigma_k$ must ensure that the vertices in $Ngb_k(w_j)$ must be contiguous and must either be flushed to the Left or to the Right.
\end{itemize}

\item \textbf{Category 2: There exists exactly two partial-block-neighbor of $w_j$.} 
Let $k_1, k_2$ be the two blocks which are partial-block-neighbors of $w_j$.  In this case we claim that the following three cases may happen:
\begin{itemize}
\item \textbf{Category 2a: $\{k_1, k_2\} = \{1, 2\}$ and  the $\s_{\widetilde{G_P}}$ begins with 121.} In this case the following two conditions must be satisfied: 
\begin{itemize}
\item \textbf{Condition 2a(1)} $\sigma_{1}$ must ensure that the vertices in $Ngb_{1}(w_j)$ are contiguous and must be either flushed to the Right or Left and  $\sigma_{2}$ must ensure that the vertices in $Ngb_{2}(w_j)$ are contiguous and must be either flushed to the Right or Left.
\item   \textbf{Condition 2a(2)} If  the vertices in $Ngb_{1}(w_j)$  flushed to the Left then the vertices of $Ngb_{2}(w_j)$ must be flushed to Right and if  the vertices in $Ngb_{1}(w_j)$  flushed to the Right then the vertices of $Ngb_{2}(w_j)$ must be flushed to the Left.
\end{itemize}
\item \textbf{Category 2b: $\{k_1, k_2\} = \{t-1, t\}$  and  the $\s_{\widetilde{G_P}}$ ends with $t(t-1)t$.} Like the previous condition, both the following two conditions must satisfied: 
\begin{itemize}
\item  \textbf{Condition 2b(1)} $\sigma_{t-1}$ must ensure that the vertices in $Ngb_{t-1}(w_j)$ are contiguous and must be either flushed to the Right or Left and 
 $\sigma_{t}$ must ensure that the vertices in $Ngb_{t}(w_j)$ are contiguous and must be either flushed to the Right or Left.
\item  \textbf{Condition 2b(2)} If  the vertices in $Ngb_{t-1}(w_j)$  flushed to the Left then the vertices of $Ngb_{t}(w_j)$ must be flushed to Right and if  the vertices in $Ngb_{t-1}(w_j)$  flushed to the Right then the vertices of $Ngb_{t}(w_j)$ must be flushed to Left
\end{itemize}
\item \textbf{Category 2c: CheckConditionB1($w_j$) outputs $(k_1, k_2)$.} In this case (from Corollary~\ref{cor:1intersection}) CheckConditionB1($w_j$) cannot output both $(k_1, k_2)$ and $(k_2, k_1)$ unless in the Category 2a and 2b.  Thus in this case $\sigma_{k_1}$ must ensure that the vertices in $Ngb_{k_1}(w_j)$ are contiguous and are flushed to the Right and  $\sigma_{k_2}$ must ensure that the vertices in $Ngb_{k_2}(w_j)$ are contiguous and are flushed to the Left. 
\end{itemize}

\item \textbf{Category 3:} If for some $w_j$, there are more than $2$ partial block neighbors then $G$ cannot be a PTPIG.
\end{itemize}

\noindent
Using algorithm CheckConditionB1($w_j$) we can identify all the various kind of restrictions on $\sigma_1, \dots, \sigma_t$ those are necessary to be imposed so that $G$ is a PTPIG.  Note that if there exists $\sigma_1, \dots, \sigma_t$ such that the above restriction are satisfied for all $w_j$, then $G$ is a PTPIG.  So our goal is to check if there exists $\sigma_1, \dots, \sigma_t$ such that the above restriction are satisfied for all $w_j$. 

\vspace{1em}\noindent
For this we define a generalization of the consecutive $1$'s problem (we call it the oriented-consecutive $1$'s problem) and show how that can be used to store all the restrictions in the PQ-tree that already stores all the restrictions which are imposed by the fact that $G_P$ is a proper interval graph. We will first show how to handle all the conditions except for the special conditions Condition 2a(2) for Category 2a and Condition 2b(2) for Category 2b. After we have successfully stored all the other restrictions we will show how to handle these special conditions.

\vspace{1em}\noindent
Note that except for these special conditions all the restrictions are of the following four kinds: 

\noindent
\begin{itemize}
\item All the vertices in $Ngb_k(w_j)$ are consecutive and flushed to the left 
\item All the vertices in $Ngb_k(w_j)$ are consecutive and flushed to the Right.   
\item All the vertices in $Ngb_k(w_j)$  are  consecutive and either flushed to the Left or flushed to the Right.
\item All the vertices in $Ngb_k(w_j)$ are  consecutive.
\end{itemize}

\noindent
In the Section~\ref{oc1p} we describe the Oriented-Consecutive $1$'s problem and show how that can be used to store all the above restrictions except for the special conditions mentioned above. The following Lemma shows how to handle the special conditions Condition 2a(2)  for Category 2a and Condition 2b(2) for Category 2b.

\begin{lem}
Let $\sigma_1, \dots, \sigma_t$ be an ordering of blocks $B_1,\ldots,B_t$ in $G$ such that for all $k=1,2,\ldots,t$ and for all $w_j\in N$, all the restrictions imposed by $w_j$ on $\sigma_k$ is satisfied, except for the special conditions Condition 2a(2) for Category 2a and Condition 2b(2) for Category 2b. If $\sigma_k^r$ is denotes the reverse permutation of $\sigma_k$ then one of the 16 possibilities obtained by reversing or not reversing the orderings $\sigma_1, \sigma_2, \sigma_{t-1}$ and $\sigma_t$ will be a valid ordering if $G$ is PTPIG.
\end{lem}

\begin{proof}
Straightforward.
\end{proof}

\noindent
So by the above lemma once we identify all the orderings $\sigma_1, \ldots, \sigma_t$ that satisfies all the restrictions imposed by all $w_j \in N$ except for the special conditions Condition 2a(2) for Category 2a and Condition 2b(2) for Category 2b, we have to check if any of the 16 possibilities is a valid set of permutation.

\subsubsection{Oriented-Consecutive $1$'s problem}\label{oc1p}

Consider the following problem: 

\vspace{1em}\noindent
\textbf{Input:} A set $\Omega = \{s_1 \dots, s_m\}$ and $n$ restrictions $(S_1, b_1), \dots, (S_n, b_n)$ where $S_i\subseteq \Omega$ and $b_i\in \{-1, 0, 1, 2\}$.

\noindent
\textbf{Question:} Is there a linear ordering of $\Omega$, say, $s_{\sigma(1)}, \dots, s_{\sigma(m)}$ such that for $1\leq i \leq n$ the following are satisfied:
\begin{itemize}
\item If $b_i = 0$, then all the elements in $S_i$ are consecutive in the linear ordering. 
\item If $b_i = -1$, then all the element in $S_i$ are consecutive in the linear ordering and all the elements of $S_i$ are flushed towards Left, i.e., $s_{\sigma(1)} \in S_i$.
\item If $b_i = 1$, then all the element in $S_i$ are consecutive in the linear ordering and all the elements of $S_i$ are flushed towards Right, i.e., $s_{\sigma(m)} \in S_i$.
\item If $b_i = 2$, then all the element in $S_i$ are consecutive in the linear ordering and all the elements of $S_i$ are either flushed towards Left or flushed towards Right, i.e., either $s_{\sigma(1)} \in S_i$ or $s_{\sigma(m)} \in S_i$.
\end{itemize}

\noindent
Now this problem is very similar to the consecutive $1$'s problem. The consecutive $1$'s problem is solved by using the PQ-tree. Recall the algorithm: given a set $\Omega = \{s_1, \dots, s_m\}$ and a set of subsets of $\Omega$, say, $S_1, \dots, S_m$, the algorithm outputs a PQ tree $T$ with leaves $\{s_1, \dots, s_m\}$ and the property that $PQ(T)$ is the set of all orderings of $\Omega$ where for all $S_i$ ($1\leq i \leq m$), the elements in $S_i$ are contiguous. The algorithm has $m$ iterations. The algorithms starts will a trivial PQ tree $T$ and at the end of the $k$-th iteration the PQ tree $T$ has the property that $PQ(T)$ is the set of all orderings of $\Omega$ where for all $S_i$ ($1\leq i \leq k-1$), the elements in $S_i$ are contiguous in the ordering. At iteration $k$, the PQ tree $T$ is updated using the algorithm $Restrict(T,S_k)$\cite{BL76}\footnote{In \cite{BL76} they define the update algorithm as $Reduce(T,S)$. In this paper we use the name $Restrict(T,S)$ instead of $Reduce(T,S)$ as we think that the word ``Restrict" is more suitable.}. If for some $S_k$, $Restrict(T,S_k)$ cannot be performed, the algorithm halts and rejects. 

Also note that the algorithm $Restrict(T, S_k)$ has linear amortized time complexity. 

\vspace{1em}\noindent
We will use similar technique to solve the Oriented-Consecutive $1$'s Problem. In fact we will reduce the Oriented-Consecutive $1$'s Problem to the standard Consecutive $1$'s Problem, but we will assume that all the restrictions of the form  $(S_i, b_i)$, where $b_i = 2$, appear at the end. That is, if for some $i$ we have $b_i=2$, then for all $j\geq i$ we will have $b_j=2$. We present Algorithm~\ref{algoPQtree} for solving the Oriented-Consecutive $1$'s Problem and the following result proves that the correctness of the algorithm. 

\begin{claim}\label{algo:PQtree}
If we assume that $b_i=2$ implies $b_j = 2$ for all $j\geq i$, then Algorithm~\ref{algoPQtree} is a correct algorithm for the Oriented-Consecutive $1$'s Problem. 
In fact all the valid ordering are stored in the PQ-tree $T$ and the amortized running time of the algorithm is linear. 
\end{claim}

\begin{proof} First let us assume that $b_i \neq 2$ for all $i$. We initialize the PQ-tree with $\Omega\cup \{\vdash, \dashv\}$ as leaves. Then we use Restrict($T, \Omega\cup \{\vdash\}$) and Restrict($T, \Omega\cup \{\dashv\}$) this implies that any ordering of the $\sigma$ of $\Omega\cup \{\vdash, \dashv\}$ must have $\vdash$ and $\dashv$ at the two ends. Note that if $\sigma$ is in fact a valid ordering of the elements in $\Omega\cup \{\vdash, \dashv\}$ that satisfies all the constraints then the reverse of $\sigma$ is also a valid ordering of the elements in $\Omega\cup \{\vdash, \dashv\}$ that satisfies all the constraints. So we can always assume that $\vdash$ is in the beginning (i.e., the leftmost position) and $\dashv$ is at the end (i.e., the rightmost position). 

\vspace{0.5em}\noindent
From now on, if we say the elements of the set $S_i$ is flushed to the Left we would mean that $S_i\cup \{\vdash\}$ is contiguous. Similarly if we say the elements of the set $S_i$ is flushed to the Right we would mean that $S_i\cup \{\dashv\}$ is contiguous. Now let us check case by case basis: 

\noindent
\begin{itemize}
\item[\textbf{Case 1:}] $b_j = 0$. Then by using Restrict($T, S_j$) we ensure that all the elements in $S_j$ are contiguous in the ordering. 
\item[\textbf{Case 2:}] $b_j = -1$. Then by using Restrict($T, S_j\cup \{\vdash\}$) we ensure that all the elements in $S_j\cup \{\vdash\}$ are contiguous in the ordering  and since $\vdash$ is at the  leftmost position so that means the elements in $S_j$ are contiguous and flushed to the Left. 
\item[\textbf{Case 3:}] $b_j = 1$. Then by using Restrict($T, S_j\cup \{\dashv\}$) we ensure that all the elements in $S_j\cup \{\dashv\}$ are contiguous in the ordering  and since $\dashv$ is at the  rightmost position so that means the elements in $S_j$ are contiguous and flushed to the Right. 
\end{itemize}

\normalsize

\noindent
So if $b_i\neq 2$ for all $i$, then the algorithm solves the Oriented-Consecutive $1$'s Problem correctly. 

\vspace{0.5em}\noindent
Now let us assume that there exists $i$ such that for all $j\geq i$, $b_j =2$ and for all $j<i$, $b_j \neq 2$.  Firstly for all $i\leq j \leq n$ we can use Restrict($T, S_j$) to ensure that
the elements of $S_j$ are contiguous. So all we have to check if they can be all flushed to either left or right. 

Note that if there $i\leq j_1, j_2 \leq n$ such that $S_{j_1}$ and $S_{j_2}$ have non zero intersection then: 
\begin{enumerate}
\item \textbf{If $S_{j_1}\subset S_{j_2}$ then} it is sufficient to ensure that $S_{j_1}$ is flushed to the left or right and hence we can forget about flushing $S_{j_2}$ to left or right. 
\item \textbf{If $S_{j_2}\subset S_{j_1}$ then} it is similar to the previous case. 
\item \textbf{If $S_{j_1}\not\subset S_{j_2}$ and $S_{j_2}\not\subset S_{j_1}$ then} note that both the conditions can be simultaneously satisfied only if $S_{j_1}\cup S_{j_2} = \Omega$ and 
in that case the following restrictions are necessary and sufficient:
\begin{enumerate}
\item $S_{j_1} \cap S_{j_2}$ is contiguous
\item $S_{j_1}\backslash S_{j_2}$ is contiguous
\item $S_{j_2}\backslash S_{j_1}$ is contiguous 
\end{enumerate}
\end{enumerate}

Note that one can in linear time can go over the sets $S_j$ and, using the above observation, reduce the set of restriction to the form ($S_j, 2$) where all the $S_j$ are disjoint.
Also note that if there are three mutually disjoints sets $S_{j_1}$, $S_{j_2}$ and $S_{j_3}$ then all three restriction $(S_{j_1}, 2)$, $(S_{j_2}, 2)$ and $(S_{j_3}, 2)$ cannot be satisfied simultaneously. 

So after the above reduction we are left with at most two restriction of the form $(S_{j_1}, 2)$ and $(S_{j_2}, 2)$ that are yet to be imposed. 

Note that till this point we have only used the Restrict algorithm on various sets and hence amortized runtime is still linear.  At this stage since only two more restrictions are to be 
imposed we can make 2 copies of the current PQ-tree and in one PQ-tree check if $Restrict(T, S_{j_1}\cup \{\vdash\})$ and $Restrict(T, S_{j_2}\cup \{\dashv\})$ can be successfully applied
and in the other PG-tree check if $Restrict(T, S_{j_1}\cup \{\dashv\})$ and $Restrict(T, S_{j_2}\cup \{\vdash\})$ can be successfully applied. 
 
If for either PQ-tree the restrictions cannot be successfully applied then we reject. Else we have a new PQ-tree. Note that if for both the PQ-trees the restrictions can be successfully
applied then if $\sigma$ is a permutation that satisfies the conditions, then the reverse of permutation $\sigma^r$ will also satisfy the conditions. Now if for some $i$, $b_i=2$ and we flushed $S_i$ to Left. 

Thus the total runtime of the algorithm is linear.  

The main idea is to consider all the restrictions with $b_i =2$ at the very end. This can be done by storing all sets $S_i$, which have the restrictions of the form $(S_i, 2)$, in a 
bucket (say $\mathcal{S}$) and all then using an algorithm (say CombinedORestrict($T, \mathcal{S}, \{\vdash, \dashv\}$)) to carefully reduce the set of restriction to at most two 
restrictions, as described above, and then checking if the two restrictions can be applied using 2 copies of the PQ-tree.    
\end{proof}

\begin{algorithm}[h]
\SetKwData{Left}{left}\SetKwData{This}{this}\SetKwData{Up}{up}
\SetKwFunction{Union}{Union}\SetKwFunction{FindCompress}{FindCompress}
\SetKwInOut{Input}{Input}\SetKwInOut{Output}{Output}
\footnotesize{
\Input{$\Omega$, $(S_1, b_1),\dots, (S_n, b_n)$ } 
\Output{A $PQ$-tree $T$ over the universe $\Omega$ such that all the remaining orderings satisfy the restrictions $(S_j, b_j)$} 
Initialize PQ-tree $T$ with $\Omega \cup \{\vdash, \dashv\}$ as leaves attached to the root (which is a P-node)\\
Initialize set $\mathcal{S}$ to $\emptyset$
Restrict($T, \Omega\cup \{\vdash\}$); \\ 
Restrict($T, \Omega\cup \{\dashv\}$); \\
\For{$j \leftarrow 1$ \KwTo $t$}{
\If{$b_j = 0$}{
Restrict($T, S_j$);
}
\If{$b_j = -1$}{
Restrict($T, S_j\cup \{\vdash\}$);
}
\If{$b_j = 1$}{
Restrict($T, S_j\cup \{\dashv\}$);
}
\If{$b_j = 2$}{
Restrict($T, S_j$);\\
$\mathcal{S} = \mathcal{S} \cup \{S_j\}$;
}
}
CombinedORestrict($T, \mathcal{S}, \{\vdash, \dashv\}$);
}
\caption{Oriented-Consecutive $1$'s test}\label{algoPQtree}
\end{algorithm}

\begin{algorithm}[h]
\SetKwData{Left}{left}\SetKwData{This}{this}\SetKwData{Up}{up}
\SetKwFunction{Union}{Union}\SetKwFunction{FindCompress}{FindCompress}
\SetKwInOut{Input}{Input}\SetKwInOut{Output}{Output}
\footnotesize{
\Input{PQ-tree $T$ on set $\Omega$, set $\mathcal{S}\subset 2^{\Omega}$, $\set{a, b}$}
\Output{PQ-Tree}
\For{$S\in \mathcal{S}$}{
\If{$Restrict(T, S\cup \{a\})$ is possible}{Restrict($T, S\cup \{a\}$)}
\Else{ Restrict($T, S\cup \{b\}$)}
}
}
\caption{CombinedORestrict($T, S, \{a,b\}$)}\label{algonew51}
\end{algorithm}

\subsubsection{The algorithm}

\noindent
The Algorithm~\ref{algonew2} ckecks if for a given $w_j\in N$ the Condition (B1) is satisfied. The algorithm is similar to the one in CASE I where we assumed that the graph $G_P$ is a connected reduced proper interval graph. Once we have the Algorithm~\ref{algonew2} for checking if the graph satisfies Condition (B1) for every $w_j \in N$, the Algorithm~\ref{algonew5} checks whether the graph $G=(P,N,E)$ is a PTPIG.  In the following theorem we prove the correctness of the algorithm.

\begin{thm}
Given graph $G=(P, N ,E)$ such that $G_P$ is a connected proper interval graph, the Algorithm \ref{algonew5} correctly decides whether $G$ is a PTPIG with probes $P$ and nonprobes $N$ in time $O(|P|+|N|+|E_2|)$, where $E_2$ is the set of edges between probes $P$ and nonprobes $N$.
\end{thm}

\begin{proof}
We follow notations, terminologies and ideas developed in the beginning of CASE II and in the `Idea of the Algorithm'. In Lines 1--3 of Algorithm \ref{algonew5}, we fix the markers $\dashv$ and $\vdash$ at the beginning and end of each block $B_k$ respectively, $k=1,2,\ldots,t$. Then for each $j=1,2,\ldots, q$, we run the algorithm. First, in Line 5, we check Condition B1 for $w_j$ by Algorithm \ref{algonew2} which takes a help from Algorithm \ref{algonew1}. Both algorithms are similar to Algorithm \ref{algonew12} and \ref{algonew11}. The difference is that Algorithm \ref{algonew2} not only ensures that the graph $G$ satisfies Condition B1 for vertex $w_j \in N$  but also return $(k_1, k_2)$ where $k_1$ and $k_2$ are the index of the start and end block of the substring that helps to satisfy the Condition B1. It also ensure that no block in between $k_1$ and $k_2$ is a partial block neighbor of $w_j$.  If $w_j$ has more than two partial-block neighbors, the Algorithm rejects. 

\vspace{1em}\noindent
Now in Lines 6--24, Algorithm \ref{algonew5} uses the ideas developed for Algorithm \ref{algoPQtree} for Oriented-consecutive $1$'s problem to apply all the necessary restrictions on the orderings of vertices in each block,  as described in Section~\ref{idea2}, except for the special conditions Condition 2a(2)  for Category 2a and Condition 2b(2) for Category 2b.

\vspace{1em}\noindent
Finally for checking whether special conditions Condition 2a(2) for Category 2a and Condition 2b(2) for Category 2b are satisfied, we require Lines 17-19 and 25-26. Suppose $w_j$ falls under Category 2a and both $B_1$ and $B_2$ are partial-block-neighbors of $w_j$. Then Algorithm \ref{algonew2} returns both $(1,2)$ and $(2,1)$. Note that we have already ensure that both $Ngb_1(w_j)$ and $Ngb_2(w_j)$ are either flushed Left or Right. If for some other $w_{j'}$ we have already flushed $Ngb_1(w_{j'})$ Left or Right then by  we had only one option when we were flushing $Ngb_1(w_j)$ to Left or Right. If for no other $w_{j'}$ we have already flushed $Ngb_1(w_{j'})$ Left or Right then we can consider either $\sigma_1$ or $\sigma_1^r$.  The same argument holds for $\sigma_2$, $\sigma_t$ and $\sigma_{t-1}$. Thus if we check for all possibilities for $\sigma_i$ or $\sigma_i^r$, $i=1,2, t, t-1$ (as in Lines 25-26 of Algorithm \ref{algonew5}) we can correctly decide whether $G$ is a PTPIG with probes $P$ and nonprobes $N$.

\vspace{1em}\noindent
Now we compute the running time. The total running time of Algorithm \ref{algonew2} (including the formation of look up tables and the function $f_{n_i}(j)$ is $O(|P|+|E_2|)$ as before, where $E_2$ is the set of edges between probes $P$ and nonprobes $N$). The variation of PQ-tree algorithm used in Lines 10, 12, 14, 16, 19 and 22 requires $O(|P|+|N|+|E_2|)$ as it requires for normal PQ-tree algorithm. The Line 26 requires constant times of $|E_2|$ steps for $4$ blocks, namely, $B_1,B_2,B_{t-1}$ and $B_t$ where $\sigma_s$ or $\sigma_s^r$ ($s=1,2,t-1,t$) is required if the special cases Category 2a or 2b occur.
Thus the total running time is $O(|P|+|N|+|E_2|)$.
\end{proof}

\begin{algorithm}[h]
\SetKwData{Left}{left}\SetKwData{This}{this}\SetKwData{Up}{up}
\SetKwFunction{Union}{Union}\SetKwFunction{FindCompress}{FindCompress}
\SetKwInOut{Input}{Input}\SetKwInOut{Output}{Output}
\footnotesize{
\Input{An array $A[1, T]$ containing only $0/1/2$ entries, $d$ (Assume $A[T+1] = 0$).}
\Output{Output all $(s, t)$ pairs such that $(t-s + 1) \geq d$ and for all $s< i < t$, $A[i] =1$,  $A[s], A[t]\in \{1,2\}$, (either $s=1$ or $A[s-1] = 0$ or $A[s-1]= A[s]=2$) and (either $t=T$ or $A[t+1]=0$ or $A[t] = A[t+1] = 2$).}
\BlankLine
Initialize $\ell = r = 0$
\While{$r\leq T$}{
\If{$A[r+1] = 0$}{
\If{$r-\ell \geq d$}{Output $(\ell+1, r)$}
$\ell = r+1$; $r= r+1$
}
\If{$A[r+1]= 1$}{$r = r+1$}
\If{$A[r+1] = 2$}{
\If{$A[r] =1$}{
\If{$r+1-\ell \geq d$}{Output $(\ell+1, r+1)$}
$\ell = r$; $r= r+1$
}
\If{$A[r] \neq 1$}{
$\ell = r$; $r = r+1$ 
}
}
}
}
\caption{Finding-1-2-Subsequence}\label{algonew1}
\end{algorithm}
\begin{algorithm}[ht]
\SetKwData{Left}{left}\SetKwData{This}{this}\SetKwData{Up}{up}
\SetKwFunction{Union}{Union}\SetKwFunction{FindCompress}{FindCompress}
\SetKwInOut{Input}{Input}\SetKwInOut{Output}{Output}
\footnotesize{
\Input{Given a graph $G=(P, N, E)$, $w_j \in N$. \textbf{(For $f_{w_j}$ refer to Definition~\ref{def:012}).}} 
\Output{If a  substring that can satisfy Condition B1 (for the vertex $w_j$) exist then it outputs $(k_i, k_2)$ where $k_1$ and $k_2$ are the block numbers that appear at the beginning and end of the substring, respectively. If no such substring exists then it outputs REJECT.}
Initialize two arrays $X$ and $Y$ of length $4d_j +1$ by setting all the entries to $0$.\\
Let $\ell = L(A_j[1])$ and $r=R(A_j[1])$. \\
\For{each $k\leftarrow 1$ \KwTo $d_j$}{\label{forins2}
Let $s = L(A_j[k])$ and $t = R(A_j[k])$\\
If ($\ell - 2d_j \leq s\leq \ell+2d_j$) \textbf{Mark}  $X[s-\ell + 2d_j +1] = f_{w_j}(A_j[k])$\\
If ($\ell - 2d_j \leq t\leq \ell+2d_j$) \textbf{Mark}  $X[t-\ell + 2d_j +1] =  f_{w_j}(A_j[k])$\\
If ($r - 2d_j \leq s\leq r+2d_i$) \textbf{Mark}  $Y[s-r + 2d_j +1] =  f_{w_j}(A_j[k])$\\
If ($r - 2d_j \leq t\leq r+2d_j$) \textbf{Mark}  $Y[t-r + 2d_j +1] =  f_{w_j}(A_j[k])$\\
}
\For{all $(a, b)$ output of \textbf{Find-1-2-Subsequence($X,d_j$)}}{
Let $\hat{a} = a + \ell - 2d_j -1$ and $\hat{b} = b + \ell - 2d_j -1$.\\
\If{$\forall 1\leq k \leq d_j$ Either $\hat{a}\leq L(A_j[k]) \leq\hat{b}$ OR  $\hat{a}\leq R(A_j[k]) \leq\hat{b}$}{
Return($k_1, k_2$) such that (Either $L(b_{k_1}) = \hat{a}$ OR $R(b_{k_1}) = \hat{a}$) AND (Either $L(b_{k_2})=\hat{b}$ OR $R(b_{k_2})=\hat{b}$)
}
}
\For{all $(a, b)$ output of \textbf{Find-1-2-Subsequence($Y,d_j$)}}{
Let $\hat{a} = a + r - 2d_j -1$ and $\hat{b} = b + r - 2d_j -1$.\\
\If{$\forall 1\leq k \leq d_j$ Either $\hat{a}\leq L(A_j[k]) \leq\hat{b}$ OR  $\hat{a}\leq R(A_j[k]) \leq\hat{b}$}{
Return($k_1, k_2$) such that (Either $L(b_{k_1}) = \hat{a}$ OR $R(b_{k_1}) = \hat{a}$) AND (Either $L(b_{k_2})=\hat{b}$ OR $R(b_{k_2})=\hat{b}$)
}
}
If nothing is returned then \textbf{REJECT}\\
}
\caption{CheckConditionB1($w_j$)}\label{algonew2}
\end{algorithm}


\noindent
\IncMargin{0em}
\begin{algorithm}
\SetKwInOut{Input}{Input}\SetKwInOut{Output}{Output}

\Input{Given a graph $G=(P,N,E)$} 
\Output{Accept if $G$ is a PTPIG}

\footnotesize{
\For{$k \leftarrow 1$ \KwTo $t$}{
Initialize PQ-tree $T_k$ with $B_k \cup \{\vdash, \dashv\}$ as leaves attached to the root (which is a P-node)\\
Restrict($T_k, B_k\cup \{\vdash\}$); \ Restrict($T_k, B_k\cup \{\dashv\}$);
Initialize set $S_k$ to $\emptyset$; 
}
\For{$j\leftarrow 1$ \KwTo $q$}{
\If{Condition B1 is satisfied by node $w_j$}{
\If{Number of Partial-Block Neighbors of $w_j$ is 0}{$j = j+1$}
\If{Number of Partial-Block-Neighbors of $w_j$ is 1 (let $B_k$ be the partial-block-neighbor)}{
\If{Number of Block-Neighbors of $w_j$ is $1$}{Restrict($T_k,  Ngb_k(w_j)$); \  $j = j +1 $}
\If{All pairs returned by $CheckConditionB1(w_j)$ are of the form $(k, k_2)$ with $k_2\neq k$}{Restrict($T_{k},  Ngb_{k}(w_j)\cup \vdash$);\ $j = j+1$}
\If{All pairs returned by $CheckConditionB1(w_j)$ are of the form $(k_1, k)$ with $k_1\neq k$}{Restrict($T_{k},  Ngb_{k}(w_j)\cup \dashv$);\ $j = j+1$}
\Else{Restrict($T_{k},  Ngb_k(w_j)$); $S_k = S_k \cup \{Ngb_k(w_j)\}$; $j = j+1$}
}
\If{Number of Block Neighbors of $w_j$ is 2 (let $B_{k_1}$, $B_{k_2}$ be the partial-block-neighbors)}{
\If{ $(\{k_1, k_2\} = \{1, 2\}$ $\&$ $\s_{\widetilde{G_P}}$ starts with $121)$ \textbf{Or}  $(\{k_1, k_2\} = \{t-1, t\}$ $\&$ $\s_{\widetilde{G_P}}$ ends with $t(t-1)t)$}
{Restrict($T_{k_1}, Ngb_{k_1}(w_j)$); Restrict($T_{k_2}, Ngb_{k_2}(w_j)$); $S_k = S_k \cup \{Ngb_{k_1}(w_j)\}\cup \{Ngb_{k_2}(w_j)\}$; $j = j+1$}
\Else{
Let $CheckConditionB1(w_j)$ returns $(k_1, k_2)$ \\ 
Restrict($T_{k_1}, Ngb_{k_1}(w_j) \cup \{\vdash\}$); Restrict($T_{k_2}, Ngb_{k_2}(w_j)\cup \{\dashv\}$); $j = j+1$}
}
\If{Number of Block Neighbors of $w_j$ is $\geq 3$}{REJECT}
}
}
}
\For{$k\leftarrow 1$ \KwTo $t$}{
CombinedORestrict($T_k, S_k, \{\vdash, \dashv\}$)
}
Let $\sigma_1, \sigma_2, \ldots, \sigma_t$ be permutations satisfying the PQ-trees $T_1, \dots, T_t$. \\
\If{$\overline{\sigma_1}, \overline{\sigma_2}, \ldots, \overline{\sigma_{t-1}}, \overline{\sigma_t}$ satisfies Condition (A) of Theorem~\ref{main}, where $\overline{\sigma_s}=\sigma_s$ or, $\sigma_s^r$ for $s=1,2,t-1,t$}{ACCEPT}
\Else{REJECT}
\caption{Testing if $G$ is PTPIG}\label{algonew5}
\end{algorithm}

\subsection{CASE III: The graph $G_P$ is a proper interval graph  (not necessarily connected or reduced)}  

\noindent
Finally, we consider the graph $G=(V,E)$ with an independent set $N$ (nonprobes) and $P=V\smallsetminus N$ (probes) such that $G_P$ is a proper interval graph, which may not be connected. Suppose $G_P$ has $r$ connected components $G_1,G_2,\ldots,G_r$ with vertex sets $P_1, \dots, P_r$. For $G$ to be a PTPIG, it is essential that the subgraphs of $G$ induced by $P_k\cup N$ is a PTPIG for each $k=1,2,\ldots, r$. As we have seen in the last subsection we can check if all the subgraphs are PTPIG in time $O(|V| + |E|)$. In fact, for each $k$, we can store all the possible canonical orderings of vertices in $P_k$ such that the corresponding canonical sequence satisfies the Condition (A) of Theorem~\ref{main} so that the graph induced by $P_k \cup N$ is a PTPIG. 

\vspace{1em}\noindent
For any vertex $w_j \in N$, let us call $G_k$ a {\em component-neighbor} of $w_j$ if there is a vertex in $P_k$ that is adjacent to $w_j$. $G_k$ is called a {\em full-component-neighbor} of $w_j$ if all the vertices in $P_k$ are neighbors of $w_j$. Also $G_k$ is called a {\em partial-component-neighbor} of $w_j$ if there exists at least one vertex in $P_k$ that is a neighbor of $w_j$ and at least one vertex in $P_k$ that is not a neighbor of $w_j$.  

\vspace{1em}\noindent
We will just be presenting the idea of the algorithm. The steps of the algorithm is obvious from the description given. We will be using all the algorithm developed so far in the previous sections for this recognition algorithm. 

\subsubsection{Idea of the algorithm}

\noindent
To check if the whole graph $G$ is a PTPIG we have to find if there exists a canonical ordering of all the vertices in $G_P$ such that for the whole graph the Condition (A) of Theorem~\ref{main} is satisfied. Note that a canonical ordering of the vertices of $G_P$ would place the vertices in each connected component next to each other and moreover for each $k$, the ordering of the vertices of $G_k$ would be a canonical ordering for the graph $G_k$. So to check if $G$ is a PTPIG we have to find if there exist an ordering of the connected components and canonical ordering of vertices in each of the components such that the corresponding canonical ordering satisfies the Condition (A) of Theorem~\ref{main}.  In fact $G$ is a PTPIG if and only if the following condition is satisfied:

\vspace{1em}\noindent
\textbf{Condition (C1):} There exists a permutation $\pi: \{1, \dots, r\} \to \{1, \dots, r\}$ and canonical sequences $\s_{G_1}, \ldots, \s_{G_r}$ of $G_1, \ldots, G_r$  such that the canonical sequence $\s_{G_P}$ of $G_P$ obtained by concatenation of the canonical sequences of $G_{\pi(1)}, \dots, G_{\pi(r)}$ (that is,  $\s_{G_P} = \s_{G_{\pi(1)}}\dots \s_{G_{\pi(r)}}$) has the property that for all $w\in N$, there exists a \textit{perfect substring} of $w$ in $\s_{G_P}$ (that is, there exists a substring of $\s_{G_P}$  where only the vertices of $w$ appear and all the neighbors of $w$ appears at least once). 

\vspace{1em}\noindent
Firstly we will use our previous algorithms to store all the possible canonical orderings of the vertices in each component so that the graphs induced by $G_k\cup N$ is a PTPIG, for each $k$. As usual we will store the restrictions using the PQ-tree. Next we will have to add some more restrictions on the canonical ordering of the vertices in each of the connected components which are necessary for the graph $G$ to be a PTPIG. These restriction will be be stored in the same the PQ-tree. At last we check if there exists an ordering of the components such that the corresponding canonical ordering satisfies the Condition (A) of Theorem~\ref{main}. 
  
\vspace{1em}\noindent
Now to find the extra restrictions on the canonical ordering we have to place we will look at several cases depending on how many partial-component neighbors are there for each $w_j\in N$. The main idea is that the following is a necessary condition for the above Condition (C1) to be satisfied:

\vspace{1em}\noindent
\textbf{Condition (C2):} If vertex $w_j$  has more than one component-neighbors, then for any component $G_k$ the canonical ordering of the vertices in $G_k$ should ensure that not only there is a perfect substring of $w_j$ in $\s_{G_k}$, but also a perfect substring of $w_j$ must be present either at the beginning or at the end of $\s_{G_k}$. 

\vspace{1em}\noindent
Depending on how many partial-component-neighbors are there for each $w_j\in N$ and the number of blocks in each of the partial component we list the restrictions on the canonical ordering of the vertices in each component so the Condition (C2) is satisfied. Here we will also use the notion of block-neighbor as defined in the previous section. 

\begin{itemize}
\item (Case 1): $w_j$ has only one component neighbor. In this case there are no extra restrictions. 
\item (Case 2): $w_j$ has no partial-component neighbors.  In this case there are no extra restrictions. 
\item (Case 3): $w_j$ has more than two partial-component-neighbors.  In this case the graph $G$ cannot be a PTPIG. 
\item (Case 4): $w_j$ has more than one component neighbors and in one of the component neighbor there are 2 or more partial-block-neighbors. In this case the graph $G$ cannot be a PTPIG. 
\item (Case 5): $w_j$ has more than one component neighbors and let $G_k$ be a partial-component neighbor of $w_j$. In this case we have two cases depending on the arrangement of blocks of $G_k$: 
\begin{itemize}
\item (Case 5a): If $G_k$ has only one block. Then we have to add the restriction that in the canonical ordering of the vertices in $G_k$ the neighbors of $w_j$ must be either flushed Left or flushed Right.  If the component $G_k$ is a partial-component-neighbor of more than one nonprobe vertex then, the choice of whether to flush Left or Right is fixed.
\item (Case 5b): If $G_k$ has more than one blocks and at least one of the block (say $B^k_{s}$) is a partial-block-neighbor of $w_j$. By Corollary~\ref{cor:disjoint} and Corollary~\ref{cor:1intersection} we see that there is at most one direction (Left or Right) to flush the neighbors of $w_j$ in $B^k_{s}$ so that in the corresponding canonical sequence there is a  perfect-substring of $w_j$ either at the beginning or the end.
\end{itemize}
\end{itemize}

\vspace{1em}\noindent
Note that one can easily identify the cases above and apply the necessary restrictions on the orderings using the PQ-tree and the same techniques as developed for the Oriented-consecutive $1$'s Problem in Section~\ref{oc1p}. If after applying the restriction we see that Condition (C2) is not satisfied, we know that $G$ is not a PTPIG.  All this can be done in time $O(|V| + |E|)$.

\vspace{1em}\noindent
Once we have been able to apply the restriction and ensure that the Condition (C2) is satisfied we now need to check if there exits an ordering of the components so that Condition (C1) is satisfied. For any $w_j\in N$ and any connected component $G_k$, if $G_k$ is a component-neighbor of $w_j$ then we know whether the perfect substring of $w_j$ is at the beginning or end of the canonical sequence of $G_k$. If the perfect substring of $w_j$ is at the beginning of the canonical sequence of $G_k$ we will call $w_j$ {\em left-oriented} with respect to $G_k$ and if the perfect substring of $w_j$ is at the end of the canonical sequence of $G_k$ we will call $w_j$ {\em right-oriented} with respect to $G_k$. 

\vspace{1em}\noindent
The problem now is to check if there exists an ordering of the connected components such that for any $w_j\in N$ the following conditions are satisfied:
\begin{enumerate}
\item the component-neighbors of $w_j$ must be consecutive,
\item there exist at most one of the component-neighbors $G_k$ of $w_j$ such that $w_j$ is right-oriented with respect to $G_k$ and this component-neighbor is first component-neighbor of $w_j$ in the ordering,
\item there exist at most one of the component-neighbors $G_k$ of $w_j$ such that $w_j$ is left-oriented with respect to $G_k$ and this component-neighbor is last component-neighbor of $w_j$ in the ordering.
\end{enumerate}

\vspace{1em}\noindent
This problem can be reduced to a consecutive $1$'s problem: Let $\Omega = \cup_{k =1}^r \{\vdash_k, \dashv_k\}$.  For any $w_j\in N$ let us define the set $T_j$ as follows: 
\begin{itemize}
\item if $G_k$ is a full-component neighbor of $w_j$ the $T_j$ contains $\vdash_k$ and  $\dashv_k$
\item if $G_k$ is a partial-component-neighbor of $w_j$ and $w_j$ is right-oriented with respect to $B_k$ then $T_j$ contains $\dashv_k$ 
\item if $G_k$ is a partial-component-neighbor of $w_j$ and $w_j$ is left-oriented with respect to $B_k$  then $T_j$ contains $\vdash_k$
\end{itemize}

\vspace{1em}\noindent
Now $G$ is a PTPIG if and only if there exists a permutation of $\Omega$ that satisfies the following properties:
\begin{itemize}
\item For all $k$, $\{\vdash_k, \dashv_k\}$ are next to each other.
\item For all $w_j\in N$, the elements in $T_j$ must be consecutive.  
\end{itemize}

\vspace{1em}\noindent
Note that the above can be tested easily using the PQ-tree in time linear in $|V|+|E|$ and with this we have the complete recognition algorithm for PTPIG.

\section{Conclusion}
The study of interval graphs was spearheaded by Benzer \cite{B} in his studies in the field of molecular biology. In \cite{Z}, Zhang introduced a generalization of interval graphs called probe interval graphs in an attempt to aid a problem called cosmid contig mapping. In order to obtain a better model another generalization of interval graphs were introduced that capture the property of overlap information, namely tagged probe interval graphs (TPIG) by Sheng, Wang and Zhang in \cite{LP}. Still there is no recognition algorithm for TPIG, in general. 

\vspace{1em}
\noindent
In this paper we characterize and obtain linear time recognition algorithm for a special class of TPIG, namely proper tagged probe interval graphs (PTPIG). The problem of obtaining a recognition algorithm for TPIG, in general is challenging and open till date. It is well known that an interval graph is a proper interval graph if and only if it does not contain $K_{1,3}$ as an induced subgraph of it. Similar forbidden subgraph characterization for PTPIG is another interesting problem.


\end{document}